\documentclass[10pt]{article}
 \usepackage{amsmath,amsfonts,amssymb,amsthm}
 \usepackage{hyperref}
 \usepackage{multirow}

 \theoremstyle{plain}
 \newtheorem{theorem}{Theorem}[section]
 \newtheorem{lemma}[theorem]{Lemma}
 \newtheorem{corollary}[theorem]{Corollary}

 \theoremstyle{definition}
 \newtheorem{definition}[theorem]{Definition}

 \theoremstyle{remark}
 \newtheorem*{remark}{Remark}

 \title{Gravitational instantons with faster than quadratic curvature decay (III)}

 \author{Gao Chen and Xiuxiong Chen}

\begin{document}

 \maketitle

 \tableofcontents

 \section{Introduction}
  This is our third paper in a series on the gravitational instantons.  In this paper, we classify ALG and ALH gravitational instantons. In ALG case, we extend Hein's construction slightly and show that it's the only ALG gravitational instanton. In ALH case, we prove a Torelli-type theorem.

  There are lots of different definitions of gravitational instantons. As in our previous work \cite{FirstPaper} \cite{SecondPaper}, we choose the following one: A noncompact complete hyperk\"ahler manifold $M$ of real dimension 4 is called a gravitational instanton if the curvature at $x$ satisfies $|\mathrm{Rm}(x)|=O(r(x)^{-2-\tau})$, where $r(x)$ denotes the distance to a fixed point in $M$, $\tau$ is any small positive number. It's worthwhile to notice that in real dimension 4, the hyperk\"ahler condition $\mathrm{hol}\subset\mathrm{Sp}(1)$ is equivalent to the Calabi-Yau condition $\mathrm{hol}\subset\mathrm{SU}(2)$.

  It's quite easy to prove the following theorem:
  \begin{theorem}
  For gravitational instanton $M$, the following conditions are equivalent:

  (1) $M$ is flat;

  (2) $M$ has trivial holonomy;

  (3) $M$ splits as $\mathbb{R}^{4-k}\times \mathbb{T}^k$, $k=0,1,2,3$.
  \label{Flat-gravitational-instanton}
  \end{theorem}

  It will be proved in Section 3. For simplicity, in this paper, we will exclude the flat gravitational instantons.

  In our first paper \cite{FirstPaper}, according to different kinds of asymptotic geometries, we classified gravitational instantons into the following categories: ALE (Asymptotically Locally Euclidean), ALF-$A_k$, ALF-$D_k$ (Asymptotically Locally Flat), ALG and ALH (``G" and ``H" are the letters after ``E" and ``F"). Its unique tangent cone at infinity is $\mathbb{C}^2/\Gamma$, $\mathbb{R}^3$, $\mathbb{R}^3/\mathbb{Z}_2$, a flat cone with cone angle $2\pi\beta$ or $\mathbb{R}^+$, respectively.

  In the ALE case, after Bando-Kasue-Nakajima's work \cite{BandoKasueNakajima} about the improvement of asymptotic rate, Kronheimer \cite{Kronheimer1} \cite{Kronheimer2} proved that any ALE gravitational instanton must be diffeomorphic to the minimal resolution $\widetilde{\mathbb{C}^2/\Gamma}$ of the quotient singularity $\mathbb{C}^2/\Gamma$, where $\Gamma$ is a finite subgroup of $\mathrm{SU}(2)$. Moreover, the Torelli theorem holds for ALE gravitational instantons.

  $H_2(\widetilde{\mathbb{C}^2/\Gamma},\mathbb{Z})$ is generated by holomorphic curves with self intersection number -2. Let $k$ be the number of generators. Then, their intersection patterns can be classified into $A_k(k\ge 1), D_k(k\ge 4), E_k(k=6,7,8)$ Dynkin diagrams. They correspond to different types of $\Gamma$.

  Later, Minerbe \cite{Minerbe} proved that the multi-Taub-NUT metric is the only ALF-$A_k$ gravitational instanton. When $k=0$, it's called the Taub-NUT metric. The Taub-NUT metric is diffeomorphic to $\mathbb{C}^2$. When $k\ge1$, the ALF-$A_k$ gravitational instanton is diffeomorphic to the ALE-$A_k$ gravitational instanton.

  In ALF-$D_k$ case, Biquard and Minerbe \cite{BiquardMinerbe} proved that $k$ must be nonnegative. Ivanov and Ro\v{c}ek \cite{IvanovRocek} conjectured a formula using generalized Legendre transform developed by Lindstr\"om and Ro\v{c}ek \cite{LindstromRocek}. This conjecture was proved by Cherkis and Kapustin \cite{CherkisKapustin}. A more explicit formula was computed by Cherkis and Hitchin \cite{CherkisHitchin}. In our second paper \cite{SecondPaper}, we proved that it's the only possible ALF-$D_k$ gravitational instanton. When $k=0$, it's called the Atiyah-Hitchin metric \cite{AtiyahHitchin}. When $k=2$, it's called the Page-Hitchin metric \cite{Page} \cite{HitchinPage}. When $k\ge 4$, the ALF-$D_k$ gravitational instanton is diffeomorphic to the ALE-$D_k$ gravitational instanton. As a corollary of the classification result in ALF case, the Torelli theorem holds for ALF gravitational instantons \cite{SecondPaper}.

  In the ALG and ALH cases, we \cite{FirstPaper} proved a compactification result and thus confirmed a conjecture of Yau \cite{Yau}:

  \begin{theorem}
  (\cite{FirstPaper})
  For any ALG or ALH gravitational instanton $M$, there exists a compact elliptic surface $\bar M$ with a meromorphic function $z:\bar M\rightarrow\mathbb{CP}^1$ whose generic fiber has genus 1. The fiber $D=\{z=\infty\}$ is regular if $M$ is ALH, while it's either regular or of type I$_0^*$, II, II$^*$, III, III$^*$, IV, IV$^*$ if $M$ is ALG. There exists an $(a_1,a_2,a_3)$ in $\mathbb{S}^2$ such that when we use $a_1I+a_2J+a_3K$ as the complex structure, $M$ is biholomorphic to $\bar M\setminus D$.
  \label{Compactification}
  \end{theorem}

  \begin{remark}
  The type of $D$ is related to the tangent cone at infinity of $M$. See the table in Definition \ref{ALG-model-definition}.
  \end{remark}

  In this paper, we will start from an improvement of the above theorem:

  \begin{theorem}
  The $(\bar M,z)$ in the above theorem must be a rational elliptic surface (See Definition \ref{Definition-rational-elliptic-surface}). Moreover, in the ALG case, $D$ can't be regular.
  \label{Main-Theorem-1}
  \end{theorem}

  In \cite{Hein}, Hein constructed lots of ALG examples. In this paper, we will slightly modify his construction and then prove that any ALG gravitational instanton must be obtained by the modified Hein's construction:

  \begin{theorem}
  (1) Let $(\bar M,z)$ be a rational elliptic surface with $D=\{z=\infty\}$ of type I$_0^*$, II, II$^*$, III, III$^*$, IV, or IV$^*$. Let $\omega^+=\omega^2+i\omega^3$ be a rational 2-form on $\bar M$ with $[D]=\{\omega^+=\infty\}$. For any K\"ahler form $\omega$ on $\bar M$, there exists a real smooth polynomial growth function $\phi$ on $M=\bar M\setminus D$ such that $(M,\omega^1=\omega+i\partial\bar\partial\phi,\omega^2,\omega^3)$ is an ALG gravitational instanton.

  (2) The form $\omega+i\partial\bar\partial\phi$ in the first part is uniquely determined by its asymptotic geometry.

  (3) Given any ALG gravitational instanton, we can write it as $(M,\omega^1,\omega^2,\omega^3)$ after a hyperk\"ahler rotation which replace $a_1I+a_2J+a_3K$ in Theorem \ref{Compactification} by $I$. Then $\omega^+=\omega^2+i\omega^3$ is a rational 2-form on $\bar M$ with $[D]=\{\omega^+=\infty\}$. There exist a K\"ahler form $\omega$ on $\bar M$ and a real smooth polynomial growth function $\phi$ on $M=\bar M\setminus D$ such that $\omega^1=\omega+i\partial\bar\partial\phi$. When $D$ is of type II$^*$, III$^*$, or IV$^*$, we may need a new choice of $\bar M$ to achieve this.
  \label{Main-Theorem-2}
  \end{theorem}

  It's interesting to notice that in \cite{BiquardMinerbe}, Biquard and Minerbe constructed ALF-$D_k (k\ge 4)$, ALG (I$_0^*$, II, III, IV) and ALH gravitational instantons on the minimal resolutions of the quotient of Taub-NUT metric by the binary dihedral group, $(\mathbb{R}^2\times\mathbb{T}^2)/\mathbb{Z}_k (k=2,6,4,3)$ or $(\mathbb{R}\times\mathbb{T}^3)/\mathbb{Z}_2$, respectively. The three cases ALG-II$^*$, ALG-III$^*$, ALG-IV$^*$ are all missing.

  When $D$ is of type I$_b (b=1,2,...9)$ or I$_b^* (b=1,2,3,4)$, Hein \cite{Hein} also constructed some hyperk\"ahler metrics on $\bar M\setminus D$. Since they don't have fast enough curvature decay rates, we exclude them from the discussion. However, they are still very important because Cherkis and Kapustin \cite{CherkisKapustinALG} predicted complete hyperk\"ahler metrics on the moduli space of periodic monopoles, which is a rational elliptic surface minus a fiber of type I$_0^*$, I$_1^*$, I$_2^*$, I$_3^*$, I$_4^*$. They are also related to the moduli space of solutions of Hitchin equations on a cylinder. Notice that they are called ALG-$D_4, D_3, D_2, D_1, D_0$ by Cherkis-Kapustin but certainly their definition is different from our definition. Thus, we suggest the notation ALG$^*$ to denote Hein's exceptional examples. In \cite{FoscoloGluing} \cite{FoscoloModulispace}, Cherkis-Kapustin's prediction was partially verified by Foscolo. He proved that the moduli space of periodic monopoles is a non-empty hyperk\"ahler manifold. However, it's still unknown whether the metric is complete or whether it's an elliptic surface.

  It's worthwhile to notice that Biquard and Boalch \cite{BiquardBoalch} proved that the moduli space of meromorphic connections on a curve is a complete hyperk\"ahler manifold. In Boalch's previous work \cite{Boalch}, he related such moduli space to the Painlev\'e equation. Following Okamoto's work \cite{OkamotoI} \cite{OkamotoII} \cite{OkamotoIII} \cite{OkamotoIV}, Sakai \cite{Sakai} related the Painlev\'e equation to a rational elliptic surface $\bar M$ minus a fiber $D$. The type of the fiber $D$ is related to a Dynkin diagram:

    \begin{center}
  \begin{tabular}{|c|c|c|c|c|c|c|c|c|c|c|}
  \hline
  I$_0^*$ & I$_1^*$ & I$_2^*$ & I$_3^*$ & I$_4^*$ & II & II$^*$ & III & III$^*$ & IV & IV$^*$\\

  $D_4$ & $D_3$ & $D_2$ & $D_1$ & $D_0$ & $E_8$ & $A_0$ & $E_7$ & $A_1$ & $E_6$ & $A_2$\\
  \hline
  \end{tabular}
  \end{center}

  It's not known whether the Biquard-Boalch's metric is ALG or ALG$^*$. However, it's known that an open part of Biquard-Boalch's metric is diffeomorphic to the corresponding ALE/ALF gravitational instanton denoted by the same Dynkin diagram. See \cite{BoalchEk} and \cite{BoalchSurvey} for details.

  In the ALH case, as a corollary of Theorem \ref{Compactification} and \ref{Main-Theorem-1}, any ALH gravitational instantons are diffeomorphic to each other. In particular, they are diffeomorphic to the minimal resolution of $(\mathbb{R}\times\mathbb{T}^3)/\mathbb{Z}_2$ by \cite{BiquardMinerbe}. The torus $\mathbb{T}^3=\mathbb{R}^3/\Lambda$ is determined by the lattice $\Lambda=\mathbb{Z}v_1\oplus\mathbb{Z}v_2\oplus\mathbb{Z}v_3$. It's easy to see that $H_{2}(\widetilde{(\mathbb{R}\times\mathbb{T}^3)/\mathbb{Z}_2},\mathbb{R})=\mathbb{R}^{11}$ is generated by three faces $F_{jk}$ spanned by $v_j$ and $v_k$ and eight rational curves $\Sigma_j$ coming from the resolution of eight orbifold points in $(\mathbb{R}\times\mathbb{T}^3)/\mathbb{Z}_2$. Using those notations, we will prove the following classification result of ALH gravitational instantons:

  \begin{theorem}
  (Torelli theorem for ALH gravitational instantons)

  Let $M$ be the smooth 4-manifold which underlies the minimal resolution of $(\mathbb{R}\times\mathbb{T}^3)/\mathbb{Z}_2$.
  Let $[\alpha^1],[\alpha^2],[\alpha^3]\in H^2(M,\mathbb{R})$ be three cohomology classes which satisfy the nondegeneracy conditions:

  (1) The integrals $f_{ijk}$ of $\alpha^i$ on the three faces $F_{jk}$ satisfy
  $$\left|
  \begin{array}{ccc}
  f_{123}&f_{131}&f_{112} \\
  f_{223}&f_{231}&f_{212} \\
  f_{323}&f_{331}&f_{312}
  \end{array}
  \right|>0;$$

  (2) For each $[\Sigma]\in H_2(M,\mathbb{Z})$ with $[\Sigma]^2=-2$, there exists $i\in\{1,2,3\}$ with $[\alpha^i][\Sigma]\not=0$.

  Then there exists on $M$ an ALH hyperk\"ahler structure such that $\Phi$ in Definition \ref{ALH-definition} can be chosen to be the identity map and the cohomology classes of the K\"ahler forms $[\omega^i]$ are the given $[\alpha^i]$. It's unique up to tri-holomorphic isometries which induce identity on $H_2(M,\mathbb{Z})$.

  Moreover, any ALH gravitational instanton must be constructed by this way.
  \label{Main-Theorem-3}
  \end{theorem}

  \begin{remark}
  Recently, Haskins, Hein and Nordstr\"om \cite{HaskinsHeinNordstrom} classified asymptotically cylindrical Calabi-Yau manifolds of complex dimension at least 3. In dimension 2, their analytic existence theorem (Theorem 4.1 of \cite{HaskinsHeinNordstrom}) still holds. However, when $\mathbb{T}^3$ doesn't split isometrically as $\mathbb{S}^1\times\mathbb{T}^2$, their geometric existence theorem (Theorem D of \cite{HaskinsHeinNordstrom}) fails due to the lack of background K\"ahler form in the cohomology class.
  \end{remark}

  \begin{remark}
  In \cite{Hein}, Hein proved that the space of ALH gravitational instantons module isometries is 30 dimensional. After adding 3 parameters of hyperk\"ahler rotations, the space of ALH gravitational instantons module tri-holomorphic isometries which induce identity on $H_{2}(M,\mathbb{Z})$ is 33 dimensional. Our Theorem \ref{Main-Theorem-3} is consistence with Hein's computation.
  \end{remark}

  It's interesting to compare Theorem \ref{Main-Theorem-3} with the Torelli theorem for ALE gravitational instantons \cite{Kronheimer1} \cite{Kronheimer2}, ALF gravitational instantons \cite{SecondPaper} as well as K3 surfaces, which was proved by Burns-Rapoport \cite{BurnsRapoport}, Todorov \cite{Todorov}, Looijenga-Peters \cite{LooijengaPeters} and Siu \cite{Siu}. It was reformulated by Besse in Section 12.K of \cite{Besse}. Anderson \cite{AndersonTorelli} also proved a version of Torelli theorem for K3 surfaces which allows orbifold singularities.

  \begin{theorem}
  (\cite{Besse})(Torelli theorem for K3 surfaces)

  Let $M$ be the smooth 4-manifold which underlies the minimal resolution of $\mathbb{T}^4/\mathbb{Z}_2$.
  Let $\Omega$ be the space of three cohomology classes $[\alpha^1],[\alpha^2],[\alpha^3]\in H^2(M,\mathbb{R})$ which satisfy the following conditions:

  (1) (Integrability) $$\int_{M}\alpha^i\wedge\alpha^j=2\delta_{ij}V.$$

  (2) (Nondegeneracy) For any $[\Sigma]\in H_2(M,\mathbb{Z})$ with $[\Sigma]^2=-2$, there exists $i\in\{1,2,3\}$ with $[\alpha^i][\Sigma]\not=0$.

  $\Omega$ has two components $\Omega^+$ and $\Omega^-$. For any $([\alpha^1],[\alpha^2],[\alpha^3])\in\Omega^+$, there exists on $M$ a hyperk\"ahler structure for which the cohomology classes of the K\"ahler forms $[\omega^i]$ are the given $[\alpha^i]$. It's unique up to tri-holomorphic isometries which induce identity on $H_2(M,\mathbb{Z})$.

  Moreover, any hyperk\"ahler structure on K3 surface must be constructed by this way.
  \label{Torelli-K3}
  \end{theorem}

  One may ask whether Torelli Theorem holds for ALG gravitational instantons. The answer is false at least when $D$ is of type II$^*$, III$^*$, or IV$^*$.

  \begin{theorem}
  When $D$ is of type II$^*$, III$^*$, or IV$^*$, there exist two different ALG gravitational instantons with same $[\omega^i]$.
  \label{Torelli-ALG}
  \end{theorem}

  In Section 3, we will study the topology of ALG and ALH gravitational instantons. In Section 4, we will prove Theorem \ref{Main-Theorem-2}. In Section 5, we will prove Theorem \ref{Torelli-ALG} and the theorem that the gluing of any ALH gravitational instanton with itself is a K3 surface. In Section 6, we will use the gluing construction in Section 5 and the Torelli theorem for K3 surfaces to prove the uniqueness part of Theorem \ref{Main-Theorem-3}. In Section 7, we will prove the existence part of Theorem \ref{Main-Theorem-3}.\\

  \noindent{\bf Acknowledgement:} We learned the idea that the gluing of any ALH gravitational instanton with itself is a K3 surface as well as some initial set-ups of the gluing construction from the lecture of Sir Simon Donaldson in the spring of 2015 at Stony Brook University. We also thank Philip Boalch, 	Lorenzo Foscolo, Hans-Joachim Hein, Robert Lazarsfeld and Dennis Sullivan for some suggestions.

\section{Definitions}

  \begin{definition}
  (Hyperk\"ahler manifold)
  A manifold $(M,g)$ is called hyperk\"ahler if there are three compatible parallel complex structures $I, J, K$ on $M$ satisfying the quaternion relationships. Any map between two hyperk\"ahler manifolds is called tri-holomorphic if it preserves $I$,$J$ and $K$.

 It's well known that $I,J,K$ are all K\"ahler structures. They determine three K\"ahler forms $\omega^1, \omega^2, \omega^3$. The $I$-holomorphic 2-form $\omega^+=\omega^2+i\omega^3$ is called the holomorphic symplectic form. Conversely, it's well known that three closed forms $\omega^i$ satisfying $\omega^i\wedge\omega^j=2\delta_{ij}V$ for some nowhere vanishing 4-form $V$ determine a hyperk\"ahler structure on $M$. In fact, given three such 2-forms, we can call the linear span of them the ``self-dual" space. The orthogonal complement of the ``self-dual" space under wedge product is called the ``anti-self-dual" space. These two spaces determine a star operator. It's well known that the star operator determines a conformal class of metrics. The conformal factor can be determined by requiring $V$ to be the volume form. Using this metric and the three forms $\omega^i$, we can determine three almost complex structures $I$, $J$ and $K$. It's easy to see that $IJ=K$ or $IJ=-K$. In the K3 case, the former happens on $\Omega^+$ (the latter happens on $\Omega^-$). In noncompact case, the former always happens if it happens on the end. By Lemma 6.8 of \cite{Hitchin}, $I,J,K$ are parallel.
  \end{definition}

  \begin{definition}
  (Hyperk\"ahler rotation)
  For any matrix $$\left(
  \begin{array}{ccc}
  a_1&a_2&a_3 \\
  b_1&b_2&b_3 \\
  c_1&c_2&c_3
  \end{array}
  \right)\in\mathrm{SO}(3),$$
  $(M,g,a_1I+a_2J+a_3K,b_1I+b_2J+b_3K,c_1I+c_2J+c_3K)$ defines another hyperk\"ahler structure on $M$. It's called the hyperk\"ahler rotation of $(M,g,I,J,K)$. After hyperk\"ahler rotation, we can with out loss of generality, assume that the complex structure $a_1I+a_2J+a_3K$ in Theorem \ref{Compactification} is actually $I$.
  \end{definition}

  \begin{definition}
  (ALG model)
  Suppose $\beta\in(0,1]$ and $\tau\in\mathbb{H}=\{\tau|\mathrm{Im}\tau>0\}$ are parameters in the following table:
  \begin{center}
  \begin{tabular}{|c|c|c|c|c|c|c|c|c|}
  \hline
  $D$ & Regular & I$_0^*$ & II & II$^*$ & III & III$^*$ & IV & IV$^*$\\

  $\beta$ & 1 & $\frac{1}{2}$ & $\frac{1}{6}$ & $\frac{5}{6}$ & $\frac{1}{4}$ & $\frac{3}{4}$ & $\frac{1}{3}$ & $\frac{2}{3}$\\

  $\tau$ & $\in\mathbb{H}$ & $\in\mathbb{H}$ & $e^{2\pi i/3}$ & $e^{2\pi i/3}$ & $i$ & $i$ & $e^{2\pi i/3}$ & $e^{2\pi i/3}$\\
  \hline
  \end{tabular}
  \end{center}
  Suppose $l>0$ is any scaling parameter. Let $E$ be the manifold obtained by identifying $(u,v)$ with $(e^{2\pi i \beta}u,e^{-2\pi i \beta}v)$ in the space $$\{(u,v)|\mathrm{arg} u\in[0,2\pi\beta],|u|\ge R\}\subset(\mathbb{C}-B_R)\times \mathbb{C}/(\mathbb{Z}l\oplus\mathbb{Z}\tau l).$$ Then there is a flat hyperk\"ahler metric $h$ on $E$ such that $\omega^1=\frac{i}{2}(du\wedge d\bar u+dv\wedge d\bar v)$ and $\omega^+=\omega^2+i\omega^3=du\wedge dv$. It's called the standard ALG model.
  \label{ALG-model-definition}
  \end{definition}

  \begin{definition}
  (ALG)
  $(M,g)$ is called ALG of order $\delta$ if there exist a bounded domain $K \subset M$, and a diffeomorphism $\Phi:E\rightarrow M \setminus K$ such that
  $$|\nabla^m(\Phi^*g-h)|=O(|u|^{-m-\delta})$$
  for some $\delta>0$ and any $m\ge 0$.
  \label{ALG-definition}
  \end{definition}

  \begin{definition}
  (ALH model)
  Let $(E,h)$ be flat product of $[R,\infty)\times\mathbb{T}^3$. Let $r$ be the coordinate of $[R,\infty)$. Let $(\theta^1,\theta^2,\theta^3)$ be the coordinates of $\mathbb{T}^3=\mathbb{R}^3/\Lambda$. Then there exists a hyperk\"haler structure on it defined by $$\mathrm{d}r=I^*\mathrm{d}\theta^1=J^*\mathrm{d}\theta^2=K^*\mathrm{d}\theta^3.$$
  It's called the standard ALH model.
  \end{definition}

  \begin{definition}
  (ALH)
  $(M,g)$ is called ALH of order $\tau$ if there exist a bounded domain $K \subset M$, and a diffeomorphism $\Phi:E\rightarrow M \setminus K$ such that
  $$|\nabla^m(\Phi^*g-h)|=O(e^{-\tau r}), |\nabla^m(\Phi^*I-I)|=O(e^{-\tau r}),$$
  $$|\nabla^m(\Phi^*J-J)|=O(e^{-\tau r}), |\nabla^m(\Phi^*K-K)|=O(e^{-\tau r})$$
  for some $\tau>0$ and any $m\ge 0$.
  \label{ALH-definition}
  \end{definition}

  \begin{remark}
  We \cite{FirstPaper} \cite{SecondPaper} proved that the order of any ALH gravitational instanton is at least $\lambda_1=2\pi\min_{\lambda\in \Lambda^*\setminus\{0\}}|\lambda|$, where $\Lambda^*=\{\lambda\in\mathbb{R}^3|< \lambda,\theta>\in\mathbb{Z}, \forall\theta\in\Lambda\}$.
  Later, in Section 5, 6, 7, we will choose a positive number $\delta<\lambda_1/100$ and use $e^{\delta r}$ as the weight function.
  \end{remark}

  \begin{definition}
  (Rational elliptic surface)
  Let $F$, $G$ be two linearly independent cubic homogenous polynomials on $\mathbb{CP}^2$. $\{F=0\}$ and $\{G=0\}$ intersect at 9 points with multiplicity. Let $\bar M$ be the blow up of $\mathbb{CP}^2$ on these 9 points, if needed repeatedly. Then $z=F/G$ is a well-defined meromorphic function on $\bar M$ whose generic fiber has genus 1. $(\bar M,z)$ is called the rational elliptic surface. It's well known that it has a global section $\sigma$ corresponding to any exceptional curve in the blowing up construction.
  \label{Definition-rational-elliptic-surface}
  \end{definition}

  \begin{definition}
  (Holomorphic radius)
  The largest $\gamma$ satisfying the following conditions is called the $C^{1,\alpha}$-holomorphic radius at $p$:

  There exist holomorphic coordinates $z^j$ on $B(p,\gamma)$ such that
  $$\frac{1}{2}\delta_{i\bar j}\le g_{i\bar j}\le 2\delta_{i\bar j}$$
  and $$\gamma^{1+\alpha}||g_{i\bar j}||_{C^{1,\alpha}}\le 1.$$

  \end{definition}

  \begin{definition}
  $\chi$ is a smooth cut-off function from $(-\infty,+\infty)$ to $[0,1]$ such that $\chi\equiv 1$ on $(-\infty,-\frac{1}{2}]$, $\chi\equiv 0$ on $[\frac{1}{2},\infty)$ and $-2<\chi'\le 0$.
  \end{definition}

  \begin{definition}
  In the ALG case, by Theorem \ref{Compactification}, (after hyperk\"ahler rotation) there exists an $I$-holomorphic function $z$ on $M$ asymptotic to $u^{1/\beta}$. The letter $z$ in this paper will always mean this function. $r$ will always mean the function $|z|^{\beta}(1-\chi(|z|^{\beta}-1))+\chi(|z|^{\beta}-1)$.
  In the ALH case, the function $r(1-\chi(r-R-1))$ can be well defined on $M$. It's still denoted by $r$.
  \end{definition}

  \section{The topology of ALG and ALH gravitational instantons}

  In this section, we will study the topology of ALG and ALH gravitational instantons.

  We start from the study of flat gravitational instantons:

  \begin{theorem}
  For gravitational instanton $M$, the following conditions are equivalent:

  (1) $M$ is flat;

  (2) $M$ has trivial holonomy;

  (3) $M$ splits as $\mathbb{R}^{4-k}\times \mathbb{T}^k$, $k=0,1,2,3$.
  \end{theorem}
  \begin{proof}
  By the arguments in the proof of Theorem 3.4 in our first paper \cite{FirstPaper}, it's easy to see that any flat gravitational instanton $M$ must have trivial holonomy. It's well known that $M$ is isometric to the Euclidean space quotient by covering transforms. However, since the holonomy is trivial, any covering transform must be a pure translation. Therefore, $M$ is isometric to the product of the Euclidean space with a flat torus. Conversely, it's trivial that (2) or (3) implies (1).
  \end{proof}

  Therefore, as mentioned in the introduction, we will assume that any gravitational instanton is non-flat.

  \begin{theorem}
  The first betti number of any ALG or ALH gravitational instanton must be 0. Moreover, in the ALG case, $D$ can't be regular.
  \label{First-betti-number}
  \end{theorem}

  \begin{proof}
  In the ALH case, Melrose's theory \cite{Melrose} works. In particular, the first cohomology group $H^{1}(M,\mathbb{R})$ is a subspace of the space of bounded harmonic 1-forms \cite{Melrose}. By Weitzenb\"ock formula, any bounded $(\mathrm{d}^*\mathrm{d}+\mathrm{d}\mathrm{d}^*)$-harmonic 1-form $\phi$ is also $\nabla^*\nabla$-harmonic. By Melrose's theory, $\nabla\phi$ decays exponentially. After integration by parts,
  $$\int_M|\nabla\phi|^2\chi(r-R)\le\int_M|\nabla\phi||\nabla\chi(r-R)|\rightarrow0,$$
  as $R\rightarrow\infty$. Therefore, $\phi$ is a parallel 1-form. If it's nonzero, the holonomy group must be trivial since the action of $\mathrm{Sp}(1)$ is free on $\mathbb{R}^4\setminus\{0\}$. It contradicts the non-flat assumption.

  In the ALG case, (after hyperk\"ahler rotation) if $D$ is regular, i.e. if $\beta=1$, the $I$-holomorphic function $z$ on $M$ is asymptotic to the function $u$ on $E$. $\nabla\mathrm{d}u=0$ on $E$, so when we go through the construction of $z$ in our first paper \cite{FirstPaper}, it's easy to see that $|\nabla\mathrm{d}z|=O(r^{-1-\epsilon})$ for any small enough $\epsilon$.
  So $$\int_M|\nabla\mathrm{d}z|^2\chi(\frac{r}{2R}-1)\le\int_M|\nabla\mathrm{d}z||\nabla\chi(\frac{r}{2R}-1)|\rightarrow0,$$
  as $R\rightarrow\infty$. As before, it contradicts the non-flat assumption.

  Therefore $\beta<1$. Inspired by Lemma 6.11 of \cite{Melrose} ,we define $$f(r)=(\frac{1}{r}+\frac{1}{4R}\chi(\frac{2R}{r}-\frac{3}{2}))^{-1}.$$
  Then $f$ is increasing. $f(r)=r$ when $r\le R$ and $\lim_{r\rightarrow\infty}f(r)=4R$.
  Let $u=re^{i\theta}$. The map $F(re^{i\theta},v)=(f(r)e^{i\theta},v)$ on $M$ is homotopic to the identity. Therefore, any smooth closed 1-form $\phi$ is cohomologous to $F^{*}\phi$. It's easy to see that $F^{*}\phi=O(r)$.

  By Theorem 4.12 of \cite{FirstPaper}, for any small positive $\epsilon$, there exists a smooth 1-form $\psi$  such that
  $$\int_M|\psi|^2r^{-8-\epsilon}+|\nabla\psi|^2r^{-6-\epsilon}+|\nabla^2\psi|^2r^{-4-\epsilon}<\infty$$ and
  $$F^{*}\phi=\mathrm{d}\mathrm{d}^*\psi+\mathrm{d}^*\mathrm{d}\psi.$$
  Since $F^{*}\phi$ is closed, it's easy to see that $\mathrm{d}^*\mathrm{d}\psi$ is a closed harmonic 1-form.

  Similar to Theorem 4.6 of \cite{SecondPaper}, the leading term of $\mathrm{d}^*\mathrm{d}\psi$ can be written as $au^{\delta}\mathrm{d}u+b\bar u^{\delta}\mathrm{d}\bar u$ for some $\delta\le 1$. To make it well-defined, $(\delta+1)\beta$ must be an integer. The first available choice is $\delta=1/\beta-1$ if $\beta\ge1/2$ or $\delta=-1$ if $\beta<1/2$. In the first case, $\mathrm{d}^*\mathrm{d}\psi-a\beta\mathrm{d}z-b\beta\mathrm{d}\bar z$ is a much smaller closed harmonic 1-form on $M$. Its order is also at most $r^{-1}$. However, by maximal principle and the Ricci flatness, any decaying harmonic 1-form on $M$ must be 0. In conclusion, $\phi$ must be exact. In other words, $H^1(M)=0$ for any ALG gravitational instanton.
  \end{proof}

  \begin{theorem}
  $\bar M$ is a rational elliptic surface.
  \end{theorem}

  \begin{proof}
  Choose a tubular neighborhood $T$ of $D$. Then $\bar M=M\cup T$. The Mayer-Vietoris sequence is
  $$H^0(M)\oplus H^0(T)\rightarrow H^0(M\cap T)\rightarrow H^1(\bar M)\rightarrow H^1(M)\oplus H^1(T).$$

  The first map is surjective, so the second map is 0. So the third map is injective. Notice that $H^1(T)=H^1(D)=0$ because $D$ is of type I$_0^*$, II, II$^*$, III, III$^*$, IV, or IV$^*$. $H^1(M)$ also vanishes by Theorem \ref{First-betti-number}. So $H^1(\bar M)=0$.

  A careful examination of our construction of $\bar M$ in \cite{FirstPaper} and Kodaira's paper \cite{KodairaEllipticSurface} yields that $\omega^+$ can be extended to a meromorphic 2-form on $\bar M$ with a pole $D$. In other words, $D$ is the anti-canonical divisor of $\bar M$. Since $D$ is homologous to another fiber of $z$, the self intersection number of $D$ is 0. In other words $c_1^2(\bar M)=c_1^2(-K)=[D]^2=0$. It's also very easy to see that $H^0(\bar M,mK)=0$ for any $m>0$. In particular the geometric genus $p_g=\mathrm{dim}H^0(\bar M,K)=0$.

  By Kodaira's classification of complex surfaces \cite{KodairaComplexSurface}, since the first betti number of $\bar M$ is even and $p_g=0$, $\bar M$ must be algebraic. By Castelnuovo theorem, $\bar M$ must be rational because $H^{0,1}(\bar M)=H^0(\bar M,2K)=0$. By Kodaira's Equation 13 in \cite{KodairaComplexSurface}, $c_1^2+\mathrm{dim}H^{0,1}+b^{-}=10p_g+9$. So $b^{-}=9$.
  By Theorem 3 of \cite{KodairaComplexSurface}, $b^{+}=1+2p_g=1$. Therefore the second betti number $b_2$ of $\bar M$ equals to 10. It's standard \cite{Naruki} to prove that $(\bar M,z)$ is a rational elliptic surface defined in Definition \ref{Definition-rational-elliptic-surface}.
  \end{proof}

  \begin{theorem}
  Given any ALH gravitational instantons $M$, there exists a diffeomorphism from the minimal resolution of $(\mathbb{R}\times\mathbb{T}^3)/\mathbb{Z}_2$ to $M$ whose restriction on $[R,\infty)\times\mathbb{T}^3$ is $\Phi$ in Definition \ref{ALH-definition}.
  \end{theorem}
  \begin{proof}
  The divisor $D$ is smooth by our construction of the compactification. So for any small enough deformation in the coefficients of $F$ and $G$, the diffeomorphism type of $M=\bar M\setminus D$ is invariant. For generic choice of coefficients of $F$ and $G$, $\{G=0\}$ is smooth and $\{F=0\}$ intersects $\{G=0\}$ in distinct points. Since the non-generic parameters have real codimension 2, generic points can be connected by paths inside the set of generic points. Therefore, it's easy to see that any ALH gravitational instantons are diffeomorphic to each other. In particular, they are diffeomorphic to the specific example of Biquard and Minerbe \cite{BiquardMinerbe} on the minimal resolution of $(\mathbb{R}\times\mathbb{T}^3)/\mathbb{Z}_2$.
  \end{proof}

  \section{Classification of ALG gravitational instantons}

  In this section, we will slightly modify Hein's result in \cite{Hein} to get Theorem \ref{Main-Theorem-2}.

  Let $\omega$ be any K\"ahler form on $\bar M$. Let $a$ be the area of each regular fiber with respect to $\omega$. Recall that for any section $\sigma'$ of $z$ on $\Delta^*=\{|z|^\beta\ge R\}$, Hein \cite{Hein} wrote down some explicit formula of the semi-flat Calabi-Yau metric $\omega_{\mathrm{sf},a}[\sigma']$ on $M|_{\Delta^*}$ whose area of each regular fiber is also $a$:

  \begin{definition} (\cite{Hein})
  Using $\sigma'$ as the zero section, $M|_{\Delta^*}$ is locally biholomorphic to
  $$M|_{U}=(U\times\mathbb{C})/(z,v)\sim(z,v+m\tau_1(z)+n\tau_2(z))$$
  for some holomorphic functions $\tau_1$ and $\tau_2$. So locally, $\omega^+=g(z)\mathrm{d}z\wedge\mathrm{d}v$ for some holomorphic function $g:U\rightarrow\mathbb{C}$. Then locally
  $$\omega_{\mathrm{sf},a}[\sigma']=i|g|^2\frac{\mathrm{Im}(\bar\tau_1\tau_2)}{a}\mathrm{d}z\wedge\mathrm{d}\bar z+
  \frac{i}{2}\frac{a}{\mathrm{Im}(\bar\tau_1\tau_2)}(\mathrm{d}v-\Gamma\mathrm{d}z)(\mathrm{d}\bar v-\bar\Gamma\mathrm{d}\bar z),$$
  where $$\Gamma(z,v)=\frac{1}{\mathrm{Im}(\bar\tau_1\tau_2)}(\mathrm{Im}(\bar\tau_1v)\frac{\mathrm{d}\tau_2}{\mathrm{d}z}
  -\mathrm{Im}(\bar\tau_2v)\frac{\mathrm{d}\tau_1}{\mathrm{d}z}).$$
  It's easy to check that $\omega_{\mathrm{sf},a}[\sigma']$ is actually a globally well-defined form.
  \end{definition}

  After that, the following theorem is essential:

  \begin{theorem}
  There exist a real smooth polynomial growth function $\phi_1$ on $M|_{\Delta^*}$ and a polynomial growth holomorphic section $\sigma'$ of $z$ over $\Delta^*$ such that $\omega_{\mathrm{sf},a}[\sigma']=\omega+i\partial\bar\partial\phi_1$.
  \label{PartialBarPartialLemma}
  \end{theorem}
  \begin{remark}
  Compared to Hein's Claim 1 in page 382 of \cite{Hein}, the key improvement in our paper is that both $\sigma'$ and $\phi_1$ grow at most polynomially.
  \end{remark}
  \begin{proof}
  As Hein did in \cite{Hein}, there exists a real 1-form $\zeta$ on $M|_{\Delta^*}$ such that $\mathrm{d}\zeta=\omega_{\mathrm{sf},a}[\sigma]-\omega$.
  Choose the map $F$ as in Theorem \ref{First-betti-number}. By Cartan's formula, the homotopy between $F$ and the identity map implies that $F^*\omega-\omega=\mathrm{d}\zeta_1$ and $F^*\omega_{\mathrm{sf},a}[\sigma]-\omega_{\mathrm{sf},a}[\sigma]=\mathrm{d}\zeta_2$ for some real polynomial growth 1-forms $\zeta_1$ and $\zeta_2$. However, $\mathrm{d}F^*\zeta=F^*\omega_{\mathrm{sf},a}[\sigma]-F^*\omega$ for some polynomial growth 1-form $F^*\zeta$. In conclusion, we can without loss of generality assume that $\zeta$ grows polynomially.

  Using $\sigma$ as the zero section, any section $\sigma'$ of $z$ can be written as $v=\sigma'(z)$ in local coordinates. Hein calculated that there exists a real 1-form $\tilde\zeta$ such that $\omega_{\mathrm{sf},a}[\sigma']-\omega_{\mathrm{sf},a}[\sigma]=\mathrm{d}\tilde\zeta$ and the (0,1)-part $\tilde\xi$ of $\tilde\zeta$ can be written as
  $$\tilde\xi=-\frac{i}{2}\frac{a}{\mathrm{Im}(\bar\tau_1\tau_2)}[\sigma'(z)(\mathrm{d}\bar v-\bar\Gamma(z,v)\mathrm{d}\bar z)-
  \frac{1}{2}\bar\Gamma(z,\sigma'(z))\mathrm{d}\bar z].$$
  Choose $\sigma'$ so that $\frac{i}{2}\frac{a}{\mathrm{Im}(\bar\tau_1\tau_2)}\sigma'$ equals to the average of the coefficient of $\mathrm{d}\bar v$ term of the (0,1)-part $\xi$ of $\zeta$ on each fiber.
  Then $\sigma'$ and $\tilde\xi$ grow polynomially. Moreover, the average of the coefficient of $\mathrm{d}\bar v$ term of $\xi+\tilde\xi$ on each fiber vanishes. So on each fiber, $\xi+\tilde\xi$ can be written as $i\bar\partial\phi_2$ by solving the $\bar\partial$-equation on each fiber. It's easy to see that $\phi_2$ also grows polynomially. So the (0,1)-form $\xi+\tilde\xi-i\bar\partial\phi_2$ can be written as $f(z,v)\mathrm{d}\bar z$. However, it's $\bar\partial$-closed, so $f(z,v)=f(z)$. By solving the $\bar\partial$-equation on $\Delta^*$, $\xi+\tilde\xi-i\bar\partial\phi_2=i\bar\partial\phi_3$ for some polynomial growth function $\phi_3(z)$. In conclusion
  $$\omega_{\mathrm{sf},a}[\sigma']-\omega=\mathrm{d}[i\bar\partial(\phi_2+\phi_3)-i\partial(\bar\phi_2+\bar\phi_3)]
  =i\partial\bar\partial(\phi_2+\bar\phi_2+\phi_3+\bar\phi_3).$$
  \end{proof}

  \begin{theorem}
  There exists a real smooth polynomial growth function $\phi_4$ such that $\omega+i\partial\bar\partial\phi_4$ is ALG and $$(\omega+i\partial\bar\partial\phi_4)^2=\frac{1}{2}\omega^+\wedge\bar\omega^+$$
  \label{Hein-Tian-Yau}
  \end{theorem}
  \begin{remark}
  Compared to Hein's Theorem 1.3 of \cite{Hein}, the key improvements in our paper are that $\phi_4$ grows polynomially and that we obtain $\frac{1}{2}\omega^+\wedge\bar\omega^+$ instead of $\frac{\alpha}{2}\omega^+\wedge\bar\omega^+$ for large enough $\alpha$.
  \end{remark}
  \begin{proof}
  To achieve this, we still introduce a real positive bump function $b$ on $\mathbb{C}$ supported in $\{R\le|z|^\beta\le4R\}$ such that $b=1$ on $\{2R\le|z|^\beta\le3R\}$. The involution with the Green function provides a real at most polynomial growth function $\phi_5$ on $\mathbb{C}$ such that $i\partial\bar\partial\phi_5=ib(z)\mathrm{d}z\wedge\mathrm{d}\bar z$.

  Now let's look at the form $\omega+i\partial\bar\partial((1-\chi(\frac{r}{R}-\frac{5}{2}))\phi_1)$. It equals to $\omega$ when $r\le2R$ and $\omega_{\mathrm{sf},a}[\sigma']$ when $r\ge3R$. On the part $2R\le r\le3R$, this form may not be positive. However, as Hein did in Claim 3 of \cite{Hein}, $$\omega_t=\omega+i\partial\bar\partial((1-\chi(\frac{r}{R}-\frac{5}{2}))\phi_1+t\phi_5)$$ is positive for large enough $t$.

  To achieve the integrability condition $\int_M(\omega_1^2-\frac{1}{2}\omega^+\wedge\bar\omega^+)=0$, we start from choosing large enough $R$ and $t$ such that $\omega_{\mathrm{sf},a}[\sigma']$ is close enough to the standard ALG model $\frac{i}{2}\frac{a}{\mathrm{Im}\tau} (\mathrm{d}z^{\beta}\wedge\mathrm{d}\bar z^{\beta}+\mathrm{d}v\wedge\mathrm{d}\bar v)$ and $\int_M(\omega_t^2-\frac{1}{2}\omega^+\wedge\bar\omega^+)>0$.
  Then we consider $$\omega_{s,t}=\omega_t-\frac{i}{4}\frac{a}{\mathrm{Im}\tau}(1-\chi(\frac{r}{R}-6))\chi(\frac{r}{R}-s)  \beta^2|z|^{2\beta-2}\mathrm{d}z\wedge\mathrm{d}\bar z.$$
  It's easy to see that for any $s\ge5$, $\omega_{s,t}$ must be positive. What's more, since $\int_M(\omega_{s,t}^2-\frac{1}{2}\omega^+\wedge\bar\omega^+)$ decreases to negative infinity when $s$ goes to infinity, by intermediate value theorem, there exists $s$ such that the integrability condition is achieved.
  By the work of Tian-Yau \cite{TianYau}, there exists a real smooth bounded function $\phi_6$ such that $(\omega_{s,t}+i\partial\bar\partial\phi_6)^2=\frac{1}{2}\omega^+\wedge\bar\omega^+$. By Proposition 2.9 of \cite{Hein}, the solution $\omega_{s,t}+i\partial\bar\partial\phi_6$ is actually ALG.
  \end{proof}

  Thus, the first part of Theorem \ref{Main-Theorem-2} has been proved. The second part is quite simple:
  \begin{theorem}
  Suppose there exist two ALG metrics $\omega_j=\omega+i\partial\bar\partial\phi_j, j=7, 8$, satisfying $\omega_7^2=\omega_8^2=\frac{1}{2}\omega^+\wedge\bar\omega^+$, $|\nabla^m(\omega_7-\omega_8)|=O(r^{-m-\delta})$ and $|\phi_j|=O(r^N)$ for all $j=7, 8, m\ge 0$ and some $\delta, N>0$. Then $\omega_7=\omega_8$.
  \end{theorem}
  \begin{proof}
  It's easy to see that $\tilde\omega=\omega+i\partial\bar\partial(\phi_7+\phi_8)/2$ also defines a K\"ahler metric which is asymptotic to the standard ALG model. Since $\tilde\omega\wedge i\partial\bar\partial(\phi_7-\phi_8)=0$,
  $\phi_7-\phi_8$ is harmonic with respect to $\tilde\omega$. By Theorem 4.13 of \cite{FirstPaper}, it's asymptotic to $a_nz^{n}+b_n\bar{z}^n$ for some constants $a_n$ and $b_n$. The difference has smaller order. Repeat the procedure until the order is smaller than 0. The maximal principle implies that the difference is 0.
  In other words, $$\phi_7-\phi_8=\sum_{k=0}^{n}a_kz^{k}+b_k\bar z^{k}.$$
  \end{proof}

  \begin{theorem}
  The third part of Theorem \ref{Main-Theorem-2} holds.
  \end{theorem}
  \begin{proof}
  Let $(M,\omega_{\mathrm{ALG}},\omega^2,\omega^3)$ be the ALG gravitational instanton in the third part of Theorem \ref{Main-Theorem-2}. Let $a$ be the area of each regular fiber with respect to $\omega_{\mathrm{ALG}}$. Now pick a K\"ahler metric $\omega_0$ on $\bar M$ whose area of each regular fiber is $a$. Then $\omega_{\mathrm{sf},a}[\sigma']=\omega_0+i\partial\bar\partial\phi_9$ for some holomorphic section $\sigma'$ on $\{|z|^\beta\ge R\}$ and some real function $\phi_9$.
  It's easy to see that Theorem \ref{PartialBarPartialLemma} also holds for $\omega_{\mathrm{ALG}}$, i.e.
  $\omega_{\mathrm{sf},a}[\sigma'']=\omega_{\mathrm{ALG}}+i\partial\bar\partial\phi_{10}$ for some holomorphic section $\sigma''$ on $\{|z|^\beta\ge R\}$.

  When $D$ is of type I$_0^*$, II, III, or IV, i.e. $\beta\le1/2$, our goal is to show that the action $T(z,v)=(z,v+\sigma''(z)-\sigma'(z))$ as well as its inverse can be extended across $D$. If it's true, then
  $$\omega_{\mathrm{ALG}}=\omega_{\mathrm{sf},a}[\sigma'']-i\partial\bar\partial\phi_{10}
  =(T^{-1})^*\omega_{\mathrm{sf},a}[\sigma']-i\partial\bar\partial\phi_{10}=(T^{-1})^*\omega_0-i\partial\bar\partial\phi_{11}.$$
  So $\omega=\omega_{\mathrm{ALG}}+i\partial\bar\partial((1-\chi(\frac{r}{R}-\frac{5}{2}))\phi_{11}+t\phi_5)$ will be the required K\"ahler form on $\bar M$ for large enough $t$.

  To understand the structure near $D$, we start from the elliptic surface over $\Delta=\{|\tilde u|\le R^{-1}\}$ constructed by $(\Delta\times\mathbb{C})/(\tilde u,v)\sim(\tilde u,v+m\tau_1(\tilde u)+n\tau_2(\tilde u))$. Take the quotient by $(\tilde u,v)\sim(e^{2\pi i\beta}\tilde u,e^{2\pi i\beta}v)$, then there are several orbifold points in the central fiber. As Kodaira did in \cite{KodairaEllipticSurface}, those orbifold points can be resolved by replacing the neighborhoods by the non-singular models $N_{+m}$ constructed in page 583 of \cite{KodairaEllipticSurface}. Then blow down exceptional curves if they exist. $M|_{\Delta^*}\cup D$ is biholomorphic to such model by the relationship $z=u^{1/\beta}=\tilde z^{-1}=\tilde u^{-1/\beta}$.

  In those coordinates, if $T$ is given by $T(\tilde u,v)=(\tilde u,v+f(\tilde u)\tilde u)$, then by the proof of Theorem \ref{PartialBarPartialLemma}, $\frac{i}{2}\frac{af(\tilde u)\tilde u}{\mathrm{Im}(\bar\tau_1\tau_2)}$ will be the average of the coefficient of $\mathrm{d}\bar v$ term of the (0,1)-part $\xi$ of $\zeta$ on each fiber, where $\zeta$ is a real smooth polynomial growth 1-form satisfying $\mathrm{d}\zeta=\omega_{\mathrm{sf},a}[\sigma']-\omega_{\mathrm{ALG}}$. By Main Theorem 1 of \cite{SecondPaper}, the difference between the two ALG metrics $\omega_{\mathrm{ALG}}$ and $\omega_{\mathrm{sf},a}[\sigma']$ is bounded by $|u|^{-2}$. By Theorem 4.4 of \cite{SecondPaper}, $\zeta=(\mathrm{d}^*\mathrm{d}+\mathrm{d}\mathrm{d}^*)\psi_1$ and $\omega_{\mathrm{sf},a}[\sigma']-\omega_{\mathrm{ALG}}=(\mathrm{d}^*\mathrm{d}+\mathrm{d}\mathrm{d}^*)\psi_2$ for some $\psi_1$ and $\psi_2$ on $M|_{\Delta^*}$. So $(\mathrm{d}^*\mathrm{d}+\mathrm{d}\mathrm{d}^*)(\mathrm{d}^*\mathrm{d}\psi_1-\mathrm{d}^*\psi_2)=0$. Therefore, the leading term of $\mathrm{d}^*\mathrm{d}\psi_1-\mathrm{d}^*\psi_2$ must be the linear combinations of $u^\delta\mathrm{d}u$, $\bar u^\delta\mathrm{d}u$, $u^{\delta}\mathrm{d}v$, $\bar u^{\delta}\mathrm{d}v$ and their conjugates. However, $\mathrm{d}(\mathrm{d}^*\mathrm{d}\psi_1-\mathrm{d}^*\psi_2)=\mathrm{d}^*\mathrm{d}\psi_2$ has small order. So if $\delta$ is large, the leading term of $\mathrm{d}^*\mathrm{d}\psi_1-\mathrm{d}^*\psi_2$ must be the linear combinations of $u^\delta\mathrm{d}u$ and its conjugate. However, such kind of term can be written as the linear combinations of $\mathrm{d}z^m$ and its conjugate for some integer $m$. We can then subtract the leading term from $\mathrm{d}^*\mathrm{d}\psi_1-\mathrm{d}^*\psi_2$ and repeat the process. Finally, it's easy to see that $f(\tilde z)=f(\tilde u)$ is bounded by $|\tilde z|^{-\epsilon}$ for any small positive $\epsilon$. By removal of singularity theorem of holomorphic functions on the punctured disc $\Delta^*$, $f(\tilde z)$ can be extended to a holomorphic function on $\Delta$.

  Therefore, the induced map of $T$ on the resolution is holomorphic outside the central curves in $N_{+m}$ and continuous across those curves. By removal of singularity theorem, it can be extended holomorphically. Then, in the blow down procedure, the induced map is holomorphic outside the blow down of the exceptional curves. By Hartog's theorem, it can be extended on $\bar M|_\Delta$. Similarly, $T^{-1}$ can also be extended.

  When $D$ is of type II$^*$, III$^*$, IV$^*$, i.e. $1/2<\beta<1$, the arguments above fail because the meromorphic function $f$ may have a pole at $\{\tilde z=0\}$ corresponding to the term $u^{1/\beta-2}\mathrm{d}u\wedge\mathrm{d}\bar v$ in the difference $\omega_{\mathrm{sf},a}[\sigma']-\omega_{\mathrm{ALG}}$. However, recall that in \cite{FirstPaper}, we used the section $\sigma$ as zero section to compactify $M$ into $\bar M$. If we use the section $\sigma+\sigma''-\sigma'$ instead, then we may get a different $\bar M$. For this new choice of $\bar M$, the form $\omega_{\mathrm{sf},a}[\sigma'']+i\partial\bar\partial\phi_{12}=\omega_{\mathrm{ALG}}+i\partial\bar\partial\phi_{13}$ is a smooth K\"ahler form on $\bar M\cap\{|\tilde z|\le R^{-1/\beta}\}$ for some real smooth polynomial growth functions $\phi_{12}, \phi_{13}$ on $M\cap\{|\tilde z|\le R^{-1/\beta}\}$. So $\omega=\omega_{\mathrm{ALG}}+i\partial\bar\partial((1-\chi(\frac{r}{R}-\frac{5}{2}))\phi_{13}+t\phi_5)$ will be the required K\"ahler form on the new choice of $\bar M$ for large enough $t$.
  \end{proof}

 \section{Gluing of ALH gravitational instantons}

  In this section, we will prove that the gluing of any ALH gravitational instanton with itself is a K3 surface. We learned this idea as well as some initial set-ups of the gluing construction \cite{Donaldson} from the lecture of Sir Simon Donaldson in the spring of 2015 at Stony Brook University. We will also construct a counterexample of ALG Torelli Theorem when $D$ is of type II$^*$, III$^*$, or IV$^*$.

  We will use the notations of Kovalev and Singer \cite{KovalevSinger} in order to apply their results.
  Pick two copies of $M$. Define $t_1=r$ on the first copy $M_1$. Define $t_2=-r$ on the second copy $M_2$. For any gluing parameters $(\rho,\Theta)\in[R+8,\infty)\times\mathbb{T}^3$, the gluing manifold $M_{\rho,\Theta}$ is defined by truncating the two manifolds at $t_j=\pm\rho$ and identifying the boundary points $(\rho,\theta)\in M_1$ with the points $(-\rho,\Theta-\theta)\in M_2$. On $M_{\rho,\Theta}$, the function $t$ is defined by $t=t_1-\rho=t_2+\rho$. The picture can be found in page 10 of \cite{KovalevSinger}.

  Our metric on $M_{\rho,\Theta}$ is slightly different from \cite{KovalevSinger}. In fact, there are three K\"ahler forms $\omega^1$, $\omega^2$ and $\omega^3$ on $M$. The closed forms
  $$\omega^i-\omega^i_{\mathrm{flat}}=a^{i}_{j}(r,\theta)\mathrm{d}r\wedge\mathrm{d}\theta^j+
  b^{i}_{jk}(r,\theta)\mathrm{d}\theta^j\wedge\mathrm{d}\theta^k$$
  are very small on $\{\rho-1\le r\le\rho+1\}$ by Definition \ref{ALH-definition}.
  Now define $$\phi^i=[\int_{\rho}^{r}a^i_j(s,\theta)\mathrm{d}s]\mathrm{d}\theta^{j}.$$
  Then $$\mathrm{d}\phi^i=a^{i}_{j}(r,\theta)\mathrm{d}r\wedge\mathrm{d}\theta^j+
  [\int_{\rho}^{r}\frac{\partial}{\partial \theta^k}a^i_j(s,\theta)\mathrm{d}s]\mathrm{d}\theta^{k}\wedge\mathrm{d}\theta^{j}$$
  $$=a^{i}_{j}(r,\theta)\mathrm{d}r\wedge\mathrm{d}\theta^j+
  [\int_{\rho}^{r}\frac{\partial}{\partial r}b^i_{jk}(s,\theta)\mathrm{d}s]\mathrm{d}\theta^{j}\wedge\mathrm{d}\theta^{k}$$
  $$=a^{i}_{j}(r,\theta)\mathrm{d}r\wedge\mathrm{d}\theta^j+
  b^i_{jk}(r,\theta)\mathrm{d}\theta^{j}\wedge\mathrm{d}\theta^{k}
  -b^i_{jk}(\rho,\theta)\mathrm{d}\theta^{j}\wedge\mathrm{d}\theta^{k}.$$
  Therefore $$\omega^i-\omega^i_{\mathrm{flat}}=\mathrm{d}\phi^i+
  b^i_{jk}(\rho,\theta)\mathrm{d}\theta^{j}\wedge\mathrm{d}\theta^{k}$$
  are cohomologous to the forms $b^i_{jk}(\rho,\theta)\mathrm{d}\theta^{j}\wedge\mathrm{d}\theta^{k}$ on $\mathbb{T}^3$.

  Notice that any closed form on $\mathbb{T}^3$ can be cohomologous to a form with constant coefficients and any 2-form with constant coefficients is invariant under the map $\theta\rightarrow\Theta-\theta$. Therefore, when we glue the part $\{\rho-1\le t_1\le\rho+1\}$ on $M_1$ with the part $\{\rho-1\le-t_2\le\rho+1\}$ on $M_2$, the difference $\omega^i_{M_2}-\omega^i_{M_1}=\mathrm{d}\psi^i$ for some small $\psi^i$. Now define the forms $\omega^i_{\rho,\Theta}$ on $M_{\rho,\Theta}$ by $$\omega^i_{\rho,\Theta}=\omega^i_{M_1}+\mathrm{d}((1-\chi(t))\psi^i)=\omega^i_{M_2}-\mathrm{d}(\chi(t)\psi^i).$$
  Then $\omega^i_{\rho,\Theta}$ are three closed forms and $|\nabla^m(\omega^i_{\rho,\Theta}-\omega^i_{M_j})|=O(e^{-\lambda_1\rho})$ for all $m\ge 0$.

  Now we can call the linear span of $\omega^i_{\rho,\Theta}$ the ``self-dual" space. The orthogonal complement of the ``self-dual" space under wedge product is called the ``anti-self-dual" space. These two spaces determine a star operator. It's well known that the star operator determines a conformal class of metrics. The the conformal factor is determined by requiring the volume form to be $\frac{1}{2}(\det(\omega^i_{\rho,\Theta}\wedge\omega^j_{\rho,\Theta}))^{1/3}$. The resulting metric is called $g_{\rho,\Theta}$. It's slightly different from \cite{KovalevSinger}, but it satisfies all the properties needed in \cite{KovalevSinger}.

  Now we define three operators on the space of self-dual 2-forms by $$P_1\phi=e^{-\delta t_1}(\mathrm{d}^*\mathrm{d}+\mathrm{d}\mathrm{d}^*)(e^{\delta t_1}\phi),$$
  $$P_2\phi=e^{-\delta t_2}(\mathrm{d}^*\mathrm{d}+\mathrm{d}\mathrm{d}^*)(e^{\delta t_2}\phi),$$ and $$P_{\rho,\Theta}\phi=e^{-\delta t}(\mathrm{d}^*\mathrm{d}+\mathrm{d}\mathrm{d}^*)(e^{\delta t}\phi),$$
  where $\delta<\lambda_1/100$ is a small positive number. It's easy to prove the following theorem:

  \begin{theorem}
  (1) $P_1,P_2,P_{\rho,\Theta}$ are Fredholm operators from $W^{k+2,2}$ to $W^{k,2}$ for any $k\ge 0$. In other words, the kernels are finite dimensional and the cokernels, i.e. the kernels of $P_1^*,P_2^*,P_{\rho,\Theta}^*$ are also finite dimensional. The range is the $L^2$-orthogonal complement of the cokernel. The operator from the $L^2$-orthogonal complement of the kernel to the range is an isomorphism.

  (2) $\ker P_1=\mathrm{span}\{e^{-\delta t_1}\omega^i\}$, $\mathrm{coker} P_1=\{0\}$,

  $\ker P_2=\{0\}$, $\mathrm{coker} P_1=\mathrm{span}\{e^{\delta t_2}\omega^i\}$,

  $\mathrm{span}\{e^{-\delta t}\omega^i_{\rho,\Theta}\}\subset\ker P_{\rho,\Theta}$, $\mathrm{span}\{e^{\delta t}\omega^i_{\rho,\Theta}\}\subset\mathrm{coker} P_{\rho,\Theta}$.
  \label{Ker-and-Coker-of-P}
  \end{theorem}

  \begin{proof}
  The first part was proved by Lockhart and McOwen in \cite{LockhartMcOwen}.
  As for the second part, on any K\"ahler manifold with K\"ahler form $\omega$, define the operator $L$ by $L\phi=\phi\wedge\omega$. Then by K\"ahler identities, $[L,\bar\partial]=0$ and $[L,\bar\partial^*]=-i\partial$.
  Therefore
  $$[L,\mathrm{d}^*\mathrm{d}+\mathrm{d}\mathrm{d}^*]=2[L,\bar\partial^*\bar\partial+\bar\partial\bar\partial^*]=2[L,\bar\partial^*]\bar\partial
  +2\bar\partial[L,\bar\partial^*]=-2i(\partial\bar\partial+\bar\partial\partial)=0.$$
  In particular for any function $f$, $$(\mathrm{d}^*\mathrm{d}+\mathrm{d}\mathrm{d}^*)(f\omega)=(-\Delta f)\omega,$$ where $\Delta f=-(\mathrm{d}^*\mathrm{d}+\mathrm{d}\mathrm{d}^*) f$ is the ordinary Laplacian operator on functions.
  On hyperk\"ahler manifolds $M_1$ and $M_2$, there are three K\"ahler structures $I$, $J$ and $K$. Therefore,
  $$\sum_{i=1}^{3}(\mathrm{d}^*\mathrm{d}+\mathrm{d}\mathrm{d}^*)(f_i\omega^i)=-\sum_{i=1}^{3}(\Delta f_i)\omega^i.$$
  In other words, the Laplacian on the self-dual forms is exactly the Laplacian on the coefficients.
  On $M_{\rho,\Theta}$, even though the metric is not hyperk\"ahler, $\omega^i_{\rho,\Theta}$ are still harmonic since they are closed self-dual forms. The second part follows directly from the two facts above.
  \end{proof}

  The most important result of \cite{KovalevSinger} is the following theorem: (Proposition 4.2 in their paper)
  \begin{theorem}
  There exists $\rho_{*}>0$ such that for all $\rho\ge\rho_{*}$, the induced map $P''_{\rho,\Theta}$ from the $L^2$-orthogonal complement of $\mathrm{span}\{\chi(t_1-\rho/2)e^{-\delta t_1}\omega^i_{\rho,\Theta}\}$ in $W^{k+2,2}$ to the $L^2$-orthogonal complement of $\mathrm{span}\{\chi(-t_2-\rho/2)e^{\delta t_2}\omega^i_{\rho,\Theta}\}$ in $W^{k,2}$ is an isomorphism and the operator norm of $[P''_{\rho,\Theta}]^{-1}$ is bounded independent of $\rho$ and $\Theta$.
  \label{KovalevSinger}
  \end{theorem}

  It's easy to prove the following lemmas in functional analysis:

  \begin{lemma}
  (1) Suppose $V=\mathrm{span}\{v_1,...v_m\}$ is a finite dimensional subspace in $W^{k,2}$ for some $k\ge 0$. Let $V^{\perp}$ be the $L^2$-orthogonal complement of $V$ in $W^{k,2}$, then $W^{k,2}=V\oplus V^{\perp}$ and
  $$||f+g||_{W^{k,2}}\le||f||_{W^{k,2}}+||g||_{W^{k,2}}\le (1+2C_1)||f+g||_{W^{k,2}}$$
  for all $f\in V$ and $g\in V^{\perp}$, where $C_1=\sup_{f\in V\setminus\{0\}}\frac{||f||_{W^{k,2}}}{||f||_{L^2}}$.

  (2) Suppose $W=\mathrm{span}\{w_1,...w_m\}$ is another subspace. Let $a_{ij}=(w_i,v_j)_{L^2}$. If the matrix $A=\{a_{ij}\}$ is invertible with $A^{-1}=\{a^{ij}\}$, then the composition of the inclusion and the projection maps $P=\mathrm{Proj}_{W^{\perp}}\circ i:V^{\perp}\rightarrow W^{\perp}$ is an isomorphism.
  What's more $$(1+C_2)^{-1}||Pf||_{W^{k,2}}\le||f||_{W^{k,2}}\le C_3||Pf||_{W^{k,2}},$$
  where $$C_2=\sup_{f\in W\setminus\{0\}}\frac{||f||_{W^{k,2}}}{||f||_{L^2}}, C_3=1+C(m)||A^{-1}||\max||v_i||_{L^2}\max||w_j||_{W^{k,2}}.$$
  \label{Projection}
  \end{lemma}

  \begin{proof}
  The proof is quite obvious. The only thing to notice is that $$P^{-1}f=f-\sum_{i,j=1}^{m}a^{ij}(f,v_i)_{L^2}w_j.$$
  \end{proof}

  The following corollary of Theorem \ref{KovalevSinger} and Lemma \ref{Projection} provides the main estimate of this section:

  \begin{corollary}
  There exists $\rho_{*}>0$ such that for all $\rho\ge\rho_{*}$, the space of harmonic self-dual 2-forms on $M_{\rho,\Theta}$ equals to $\mathcal{H}^+=\mathrm{span}\{\omega^i_{\rho,\Theta}\}$. The Laplacian operator $\Delta_{\rho,\theta}=\mathrm{d}^*\mathrm{d}+\mathrm{d}\mathrm{d}^*$ from the $L^2$-orthogonal complement of $\mathcal{H}^+$ in $W^{k+2,2}(\Lambda^+)$ to the $L^2$-orthogonal complement of $\mathcal{H}^+$ in $W^{k,2}(\Lambda^+)$ is an isomorphism and the operator norm of $G_{\rho,\theta}=\Delta_{\rho,\theta}^{-1}$ is bounded by $Ce^{2\delta\rho}$ for some constant $C$ independent of $\rho$ and $\Theta$.
  \label{Main-estimate}
  \end{corollary}

  \begin{proof}
  The isomorphism map in Theorem \ref{KovalevSinger} can be decomposed into the composition of following maps:
  $$(\chi(t_1-\rho/2)e^{-\delta t_1}\omega^i_{\rho,\Theta})^{\perp}\rightarrow(e^{-\delta t}\omega^i_{\rho,\Theta})^{\perp} \rightarrow(\ker P_{\rho,\Theta})^{\perp}$$
  $$\rightarrow(\mathrm{coker} P_{\rho,\Theta})^{\perp}\rightarrow(e^{\delta t}\omega^i_{\rho,\Theta})^{\perp}\rightarrow(\chi(-t_2-\rho/2)e^{\delta t_2}\omega^i_{\rho,\Theta})^{\perp}.$$
  The first and the fifth maps are isomorphisms by Lemma \ref{Projection}. Therefore, the second map must be injective and the fourth map must be surjective. In other words, $\ker P_{\rho,\Theta}=\mathrm{span}\{e^{-\delta t}\omega^i_{\rho,\Theta}\}$ and $\mathrm{coker} P_{\rho,\Theta}=\mathrm{span}\{e^{\delta t}\omega^i_{\rho,\Theta}\}$. So all the maps are actually isomorphisms. By Theorem \ref{KovalevSinger} and Lemma \ref{Projection}, the operator norm of the inverse of the map $P_{\rho,\Theta}:(\ker P_{\rho,\Theta})^{\perp}\rightarrow(\mathrm{coker} P_{\rho,\Theta})^{\perp}$ is bounded. It's straight forward to switch this estimate into the estimate of the Laplacian operator.
  \end{proof}

  We are ready for the main theorem of this section:

  \begin{theorem}
  Fix $k\ge 3$. For large enough $\rho_*$ and any $\rho\ge\rho_*$, there exists a hyperk\"ahler structure $\tilde\omega^i_{\rho,\Theta}$ on $M_{\rho,\Theta}$ such that
  $||\tilde\omega^i_{\rho,\Theta}-\omega^i_{\rho,\Theta}||_{W^{k,2}}\le Ce^{(-\lambda_1+2\delta)\rho}$ for some constant $C$ independent of $\rho$ and $\Theta$.
  \label{Gluing}
  \end{theorem}

  \begin{proof}
  Fix the volume form $V=\frac{1}{2}\det(\omega^i_{\rho,\Theta}\wedge\omega^j_{\rho,\Theta})^{1/3}$ on $M_{\rho,\Theta}$. When two symmetric matrices $A=\frac{\omega^i_{\rho,\Theta}\wedge\omega^j_{\rho,\Theta}}{2V}$ and $B$ are close enough to the identity matrix, the equation $CAC^{\mathrm{T}}=B$ has a solution $C=B^{1/2}A^{-1/2}$. Define $F^i(B)$ by $F^i(B)=C_{ij}\omega^j_{\rho,\Theta}$. Then $F^i(B)\wedge F^j(B)=2b_{ij}V$.

  Recall that the map $G_{\rho,\Theta}$ on $(\mathcal{H}^+)^{\perp}\subset\Lambda^+$ satisfies
  $$\psi=(\mathrm{d}^*\mathrm{d}+\mathrm{d}\mathrm{d}^*)G_{\rho,\Theta}\psi
  =-(*\mathrm{d}*\mathrm{d}+\mathrm{d}*\mathrm{d}*)G_{\rho,\Theta}\psi$$
  $$=-(*+\mathrm{Id})\mathrm{d}*\mathrm{d}G_{\rho,\Theta}\psi=\mathrm{d}^+(-2*\mathrm{d}G_{\rho,\Theta}\psi).$$
  So if $\phi^i\in\Lambda^1$ satisfy the equation
  $$\phi^i=-2*\mathrm{d}G_{\rho,\Theta}\mathrm{Proj}_{(\mathcal{H}^+)^{\perp}} F^i(\delta_{\alpha\beta}-\frac{\mathrm{d}^-\phi^\alpha\wedge\mathrm{d}^-\phi^\beta}{2V}),$$
  then the closed forms $\tilde\omega^i_{\rho,\Theta}=\mathrm{d}\phi^i+\mathrm{Proj}_{\mathcal{H}^+} F^i(\delta_{\alpha\beta}-\frac{\mathrm{d}^-\phi^\alpha\wedge\mathrm{d}^-\phi^\beta}{2V})$
  will satisfy the required equation
  $\tilde\omega^i_{\rho,\Theta}\wedge\tilde\omega^j_{\rho,\Theta}=2\delta_{ij}V$.

  We will solve the equation by iterations
  $$\phi_0^i=0,$$
  $$\phi_{n+1}^i=-2*\mathrm{d}G_{\rho,\Theta}\mathrm{Proj}_{(\mathcal{H}^+)^{\perp}} F^i(\delta_{\alpha\beta}-\frac{\mathrm{d}^-\phi_{n}^\alpha\wedge\mathrm{d}^-\phi_{n}^\beta}{2V}).$$
  Since $W^{k,2}$ embeds into $C^{0}$, if $||\phi_{n}^i||_{W^{k+1,2}}\le e^{-\lambda_1\rho/2}$ and $\rho\ge\rho_*$, then
  $$||\frac{\omega^i_{\rho,\Theta}\wedge\omega^j_{\rho,\Theta}}{2V}-\delta_{ij}||_{C^0}
  +||\frac{\mathrm{d}^-\phi_{n}^i\wedge\mathrm{d}^-\phi_{n}^j}{2V}||_{C^0}\le Ce^{-\lambda_1\rho}$$
  can be arbitrarily small if $\rho_*$ is large.
  So $$||\phi_{n+1}^i||_{W^{k+1,2}}\le Ce^{2\delta\rho}||\mathrm{Proj}_{(\mathcal{H}^+)^{\perp}}F^i (\delta_{\alpha\beta}-\frac{\mathrm{d}^-\phi_{n}^\alpha\wedge\mathrm{d}^-\phi_{n}^\beta}{2V})||_{W^{k,2}}\le Ce^{(-\lambda_1+2\delta)\rho}.$$
  As long as $\rho_*$ is large enough, the above estimate holds by induction.
  It follows that $$||\phi_{n+2}^i-\phi_{n+1}^i||_{W^{k+1,2}}\le Ce^{(-\lambda_1+4\delta)\rho}||\phi_{n+1}^i-\phi_{n}^i||_{W^{k+1,2}}.$$
  As long as $\rho_*$ is large enough, $\phi^i=\lim_{n\rightarrow\infty}\phi_{n}^i$ will be the solution.
  \end{proof}

  \begin{corollary}
  For any ALH gravitational instanton $M$, $\int_{M}|Rm|^2=96\pi^2$.
  \label{Curvature-L2-bound}
  \end{corollary}
  \begin{proof}
  It's easy to deduce this conclusion from the well known fact that for K3 surface $M_{\rho,\Theta}$, $\int_{M_{\rho,\Theta}}|Rm(\tilde\omega^i_{\rho,\Theta})|^2=8\pi^2\chi(M_{\rho,\Theta})=192\pi^2.$
  \end{proof}

  \begin{theorem}
  Suppose $\alpha^i$ are three 2-forms satisfying the following conditions:

  (1)$$\int_{M_{\rho,\Theta}}\alpha^2\wedge\alpha^3=\int_{M_{\rho,\Theta}}\alpha^3\wedge\alpha^1=\int_{M_{\rho,\Theta}}\alpha^1\wedge\alpha^2=0,$$
  $$\int_{M_{\rho,\Theta}}\alpha^1\wedge\alpha^1=\int_{M_{\rho,\Theta}}\alpha^2\wedge\alpha^2=\int_{M_{\rho,\Theta}}\alpha^3\wedge\alpha^3;$$

  (2)$$||\alpha^i-\tilde\omega^i_{\rho,\Theta}||_{L^2}\le e^{-3\lambda_1\rho/4}.$$
  Then there exists a hyperk\"ahler structure $\omega^i$ on $M_{\rho,\Theta}$ such that $\omega^i\in[\alpha^i]$ and $||\omega^i-\tilde\omega^i_{\rho,\Theta}||_{W^{k,2}}\le Ce^{2\delta\rho}||\alpha^i-\tilde\omega^i_{\rho,\Theta}||_{L^2}.$
  \label{Quantitative-local-Torelli}
  \end{theorem}

  \begin{proof}
  Using $\tilde\omega^i_{\rho,\Theta}$ as the background hyperk\"ahler structure, we can choose harmonic representatives $\beta^i$ from the cohomology classes $[\alpha^i]$. Therefore, $$C^{-1}||\beta^i-\tilde\omega^i_{\rho,\Theta}||_{W^{k,2}}\le
  ||\beta^i-\tilde\omega^i_{\rho,\Theta}||_{L^2}\le||\alpha^i-\tilde\omega^i_{\rho,\Theta}||_{L^2}\le e^{-3\lambda_1\rho/4}.$$
  After replacing $\omega^i_{\rho,\Theta}$ by $\beta^i$ in the proof of Theorem \ref{Gluing}, we can find a hyperk\"ahler structure $\omega^i_0$ such that
  $\mathrm{span}\{[\omega^i_0]\}=\mathrm{span}\{[\alpha^i]\}$. By the condition (1) of $\alpha^i$, the hyperk\"ahler structure $\omega^i$ can be chosen to be some rescaling and hyperk\"ahler rotation of $\omega^i_0$.
  \end{proof}

  Now we will use the method in this section to construct a counterexample of Torelli Theorem in ALG case:
  \begin{theorem}
  When $D$ is of type II$^*$, III$^*$, or IV$^*$, there exist two different ALG gravitational instantons with same $[\omega^i]$.
  \end{theorem}
  \begin{proof}
  In Example 3.1 of \cite{Hein}, Hein explains how the pairs (IV,IV$^*$) occur in rational elliptic surfaces $\bar M$ birational to $(\mathbb{P}^1\times\mathbb{T}^2)/\Gamma$ with $\Gamma=\mathbb{Z}_3$. Let $D$ be the fiber of type IV$^*$. Then, the construction in \cite{Hein} provides an ALG gravitational instanton $\omega^i$ on $M=\bar M\setminus D$. Moreover, the asymptotic rate is $2+\frac{1}{\beta}$. In particular, $|\mathrm{Rm}|=O(r^{-\frac{1}{\beta}-4})$. There is a similar example when $D$ is of type II$^*$ or III$^*$.

  By Theorem 4.12 of \cite{FirstPaper}, there exists a harmonic (0,1) form $h$ on $M$ asymptotic to $\frac{1}{\frac{1}{\beta}-1}u^{1/\beta-1}\mathrm{d}\bar v$. So $\mathrm{d}(\mathrm{Re}h)$ is an exact harmonic form asymptotic to $\mathrm{Re}(u^{1/\beta-2}\mathrm{d}u\wedge\mathrm{d}\bar v)$. Moreover, it's anti-self-dual because the coefficients of its self-dual part are decaying harmonic and thus 0.

  Let's use the notations in \cite{FirstPaper}. For example, $$||\phi||_{H^2_{\delta}}=\sqrt{\int_M|\phi|^2r^\delta\mathrm{dVol}+\int_M|\nabla\phi|^2r^{\delta+2}\mathrm{dVol}
  +\int_M|\nabla^2\phi|^2r^{\delta+4}\mathrm{dVol}}.$$
  Then by Theorem 4.12 of \cite{FirstPaper}, for $k\ge 5$ and small positive $\epsilon$, there exists a map $G:H^k_{6-\frac{4}{\beta}-\epsilon}(\Lambda^+)\rightarrow H^{k+2}_{2-\frac{4}{\beta}-\epsilon}(\Lambda^+)$ such that $\psi=(\mathrm{d}^*\mathrm{d}+\mathrm{d}\mathrm{d}^*)G\psi$. We still define $F^i:\Gamma(\mathbb{R}^{3\times3})\rightarrow\Lambda^+$ as in Theorem \ref{Gluing} so that $F^i(B)\wedge F^j(B)=2b_{ij}V$. Then we do the iteration
  $$\phi_0^1=\phi_0^3=0, \phi_0^2=t\mathrm{Re}h,$$
  $$\phi_{n+1}^i=-2*\mathrm{d}G(F^i(\delta_{\alpha\beta}
  -\frac{\mathrm{d}^-\phi_{n}^\alpha\wedge\mathrm{d}^-\phi_{n}^\beta}{2V})-\omega^i)+\delta_{i2}t\mathrm{Re}h$$
  When $t$ is small enough, $(\phi_n^1,\phi_n^2-t\mathrm{Re}h,\phi_n^3)\rightarrow(\phi^1,\phi^2-t\mathrm{Re}h,\phi^3)\in H^{k+1}_{4-\frac{4}{\beta}-\epsilon}$.
  Then $\omega^i+\mathrm{d}\phi^i$ will be an ALG gravitational instanton. By direct computation, the curvature of the metric corresponding to $(J,\omega^2+t\mathrm{d}(\mathrm{Re}h))$ is proportional to $r^{\frac{1}{\beta}-4}$. It's also true for the metric corresponding to $\omega^i+\mathrm{d}\phi^i$ because their difference is in $H^k_{6-\frac{4}{\beta}-\epsilon}$. In particular, the metric corresponding to $\omega^i+\mathrm{d}\phi^i$ is not isometric to the metric corresponding to $\omega^i$.
  \end{proof}

  \section{Uniqueness of ALH gravitational instantons}

  In this section, we will prove the uniqueness part of Theorem \ref{Main-Theorem-3}.
  We start from the understanding of the cross section:

  \begin{theorem}
  The integrals of $\omega^i$ on the three faces determine the torus $\mathbb{T}^3$.
  \label{Cross-section}
  \end{theorem}

  \begin{proof}
  On the flat model, recall that
  $$\mathrm{d}r=I^*\mathrm{d}\theta^1=J^*\mathrm{d}\theta^2=K^*\mathrm{d}\theta^3.$$
  So $$\omega^1=\mathrm{d}r\wedge\mathrm{d}\theta^1+\mathrm{d}\theta^2\wedge\mathrm{d}\theta^3,$$
  $$\omega^2=\mathrm{d}r\wedge\mathrm{d}\theta^2+\mathrm{d}\theta^3\wedge\mathrm{d}\theta^1,$$
  $$\omega^3=\mathrm{d}r\wedge\mathrm{d}\theta^3+\mathrm{d}\theta^1\wedge\mathrm{d}\theta^2.$$
  The torus $\mathbb{T}^3=\mathbb{R}^3/\Lambda$ is determined by the lattice
  $\Lambda=\mathbb{Z}v_1\oplus\mathbb{Z}v_2\oplus\mathbb{Z}v_3$.
  Let $v_{ij}$ be the $\frac{\partial}{\partial\theta^j}$ components of $v_i$.
  Then   $$\left(
  \begin{array}{ccc}
  f_{123}&f_{131}&f_{112} \\
  f_{223}&f_{231}&f_{212} \\
  f_{323}&f_{331}&f_{312}
  \end{array}
  \right)=\left(
  \begin{array}{ccc}
  v_{22}v_{33}-v_{23}v_{32}&v_{32}v_{13}-v_{33}v_{12}&v_{12}v_{23}-v_{13}v_{22} \\
  v_{23}v_{31}-v_{21}v_{33}&v_{33}v_{11}-v_{31}v_{13}&v_{13}v_{21}-v_{11}v_{23} \\
  v_{21}v_{32}-v_{22}v_{31}&v_{31}v_{12}-v_{32}v_{11}&v_{11}v_{22}-v_{12}v_{21}
  \end{array}
  \right)$$
  is exactly the adjunct matrix $\mathrm{adj}(A)$ of
  $$A=\left(
  \begin{array}{ccc}
  v_{11}&v_{12}&v_{13} \\
  v_{21}&v_{22}&v_{23} \\
  v_{31}&v_{32}&v_{33} \\
  \end{array}
  \right).$$
  Since $\mathrm{adj}(A)A=\det(A)I$, it's easy to see that $\det(\mathrm{adj}(A))=(\det(A))^2$.
  Thus, $A=(\det(\mathrm{adj}(A)))^{-1/2}\mathrm{adj}(\mathrm{adj}(A))$ is determined by $\mathrm{adj}(A)$.
  On $M$, the hyperk\"ahler structure is asymptotic to the flat model. So we can get the same conclusion.
  \end{proof}

  \begin{theorem}
  ALH gravitational instantons are uniquely determined by their three K\"ahler classes $[\omega^i]$ up to tri-holomorphic isometry which induces identity on $H_{2}(M,\mathbb{Z})$.
  \label{Injectivity-of-period-map}
  \end{theorem}
  \begin{proof}
  Suppose two ALH hyperk\"ahler structures $\omega^{k,i}, k=1,2$ on $M$ satisfy $[\omega^{1,i}]=[\omega^{2,i}]$. By the results in the previous section, we have two families of K3 surfaces $(M_{\rho_k,\Theta_k},\tilde\omega^{k,i}_{\rho_k,\Theta_k})$. To understand the relationship between $M$ and $M_{\rho_k,\Theta_k}$, let's start from the flat orbifold $(\mathbb{R}\times\mathbb{T}^3)/\mathbb{Z}_2$. Take two copies of it. On the first copy, define $t_1=r$. On the second copy, define $t_2=-r$. Now we glue them by truncating the two manifolds at $t_j=\pm\rho$ and identifying the boundary points $(\rho,\theta)$ with the points $(-\rho,\Theta-\theta)$. An alternating way to describe the gluing is to start from $[0,2\rho]\times\mathbb{T}^3$, and then identify $(0,\theta)$ with $(0,-\theta)$ and identify $(2\rho,\theta)$ with $(2\rho,2\Theta-\theta)$. Let $\mathbb{T}^4=(\mathbb{R}\times\mathbb{T}^3)/\mathbb{Z}(4\rho,2\Theta)$. Then it's easy to see that the gluing is actually the orbifold $\mathbb{T}^4/\mathbb{Z}_2$.

  The resolution of this picture provides the topological picture of the construction of $M_{\rho_k,\Theta_k}$.
  The second homology group $H_2(M_{\rho_k,\Theta_k},\mathbb{R})=\mathbb{R}^{22}$ is generated by 16 curves $\Sigma_\alpha$ corresponding to 16 orbifold points, 3 faces $F_{\alpha\beta}$ spanned by $v_\alpha$ and $v_\beta$ and 3 faces $F_\alpha$ spanned by $(4\rho,2\Theta)$ and $(0,v_\alpha)$. Any hyperk\"ahler structure $\omega^i$ on the K3 surface determines 48 integrals $c_{i\alpha}$ on $\Sigma_\alpha$, 9 integrals $f_{i\alpha\beta}$ on $F_{\alpha\beta}$ and 9 integrals $f_{i\alpha}$ on $F_\alpha$. The integrability condition $\int_{M}\omega^i\wedge\omega^j=2\delta_{ij}V$ is equivalent to
  $$-\frac{1}{2}\sum_{\alpha=1}^{16}c_{i\alpha}c_{j\alpha}
  +\frac{1}{2}\sum_{(\alpha,\beta,\gamma)=(1,2,3),(2,3,1),(3,1,2)}f_{i\alpha}f_{j\beta\gamma}
  +f_{j\alpha}f_{i\beta\gamma}=2\delta_{ij}V.$$
  If $c_{i\alpha}$, $f_{i\alpha\beta}$ are given, it's a rank 5 linear system in 9 variables $f_{i\alpha}$.

  By the construction of $\tilde\omega^{k,i}_{\rho_k,\Theta_k}$ on $M_{\rho_k,\Theta_k}$, the differences $c^{1,\rho_1,\Theta_1}_{i\alpha}-c^{2,\rho_2,\Theta_2}_{i\alpha}$ and  $f^{1,\rho_1,\Theta_1}_{i\alpha\beta}-f^{2,\rho_2,\Theta_2}_{i\alpha\beta}$ are all bounded by $Ce^{(-\lambda_1+2\delta)\rho_1}+Ce^{(-\lambda_1+2\delta)\rho_2}$ for large enough $\rho_k$. However $f^{1,\rho_1,\Theta_1}_{i\alpha}-f^{2,\rho_2,\Theta_2}_{i\alpha}$ may be very large. Fortunately, we are free to change the 8 parameters $\rho_k,\Theta_k$. When $\rho_k$ and $\Theta_k$ are changed by adding $\delta\rho_k$ and $\delta\Theta_k$, the integrals $f^{k,\rho_k,\Theta_k}_{i\alpha}$ are changed by adding the almost linear terms $L(\delta\rho_k,\delta\Theta_k)+O(e^{(-\lambda_1+3\delta)\rho_1}+e^{(-\lambda_1+3\delta)\rho_2})$,
  where $$L(\delta\rho_k,\delta\Theta_k)=4\delta\rho_k(v_\alpha,\frac{\partial}{\partial\theta^i})
  +2(\delta\Theta_k\times v_\alpha,\frac{\partial}{\partial\theta^i})$$
  are determined by the cross section $\mathbb{R}^3/(\mathbb{Z}v_1\oplus\mathbb{Z}v_2\oplus\mathbb{Z}v_3)$, $\delta\rho_k$ and $\delta\Theta_k$.
  The image of $L$ is exactly the linear space
  $$\{(f_{i\alpha})| \exists C \mathrm{s.t.} \sum_{(\alpha,\beta,\gamma)=(1,2,3),(2,3,1),(3,1,2)}f_{i\alpha}f_{j\beta\gamma}
  +f_{j\alpha}f_{i\beta\gamma}=2\delta_{ij}C\},$$
  where $f_{i\alpha\beta}=(v_\alpha\times v_\beta,\frac{\partial}{\partial\theta^i})$. Therefore, after increasing the gluing parameters $\rho_1$ or $\rho_2$ and changing the parameters $\Theta_k$, $$\min_{\phi\in[\tilde\omega^{1,i}_{\rho_1,\Theta_1}-\tilde\omega^{2,i}_{\rho_2,\Theta_2}]}||\phi||_{L^2}\le Ce^{(-\lambda_1+4\delta)\rho_1}+Ce^{(-\lambda_1+4\delta)\rho_2}.$$

  By Theorem \ref{Quantitative-local-Torelli}, there exists a hyperk\"ahler structure $\omega^i$ on $M_{\rho_2,\Theta_2}$ such that $[\omega^i]=[\tilde\omega^{1,i}_{\rho_1,\Theta_1}]$ and $||\omega^i-\tilde\omega^{2,i}_{\rho_2,\Theta_2}||_{W^{k,2}}\le Ce^{(-\lambda_1+6\delta)\rho_1}+Ce^{(-\lambda_1+6\delta)\rho_2}$. By Theorem \ref{Torelli-K3}, $\omega^i$ and $\tilde\omega^{1,i}_{\rho_1,\Theta_1}$ are tri-holomorphically isometric to each other. Moreover the isometry induces identity on $H_2(M_{\rho_k,\Theta_k},\mathbb{Z})$. Notice that the long neck regions are almost flat but by Corollary \ref{Curvature-L2-bound}, the compact parts are not flat. So all the isometrics must map compact parts to compact parts. In particular, we can apply the Arzela-Ascoli theorem and the diagonal argument to get a limiting tri-holomorphic isometry on the original manifold $M$ which induces identity on $H_2(M,\mathbb{Z})$ when $\rho_k$ go to infinity.
  \end{proof}

  \begin{theorem}
  The K\"ahler classes $[\omega^i]$ satisfy the two conditions in Theorem \ref{Main-Theorem-3}.
  \end{theorem}
  \begin{proof}
  The first condition is a trivial consequence of $\det(\mathrm{adj}(A))=(\det(A))^2$ and $\det(A)\not=0$ in the proof of Theorem \ref{Cross-section}.
  As for the second condition, any ALH gravitational instanton $M$ can be glued with itself to obtain a K3 surface. By Theorem \ref{Quantitative-local-Torelli}, we can modify the hyperk\"ahler metric on the K3 surface so that the integrals of $\omega^i$ on the 11 cycles are unchanged in the gluing process. For any $[\Sigma]\in H_2(M,\mathbb{Z})$ such that $[\Sigma]^2=-2$, we can find a corresponding element in the second homology group of the K3 surface. By Theorem \ref{Torelli-K3}, there exists $i$ such that $[\omega^i][\Sigma]\not=0$. Since the integrals of $\omega^i$ on the K3 surface are the same as the integrals on $M$, the second condition must be satisfied.
  \end{proof}

  \section{Existence of ALH gravitational instantons}

  In this section, we will use the continuity method to prove the existence part of Theorem \ref{Main-Theorem-3}. Given any three classes $[\alpha_1^i]$ satisfying two conditions in Theorem \ref{Main-Theorem-3}, the cross section $\mathbb{T}^3$ is determined by Theorem \ref{Cross-section}. By the work of Biquard and Minerbe \cite{BiquardMinerbe}, there exists an ALH hyperk\"ahler structure $\omega_0^i$ on $\widetilde{(\mathbb{R}\times\mathbb{T}^3)/\mathbb{Z}_2}$. Now we are going to connect $[\alpha_1^i]$ with $[\alpha_0^i]=[\omega_0^i]$. We require that along the path, the cross section $\mathbb{T}^3$, i.e. the integrals on the faces $F_{jk}$ are invariant.

  We already know that for any $k=0,1$, any $[\Sigma]\in H_2(M,\mathbb{Z})$ with $[\Sigma]^2=-2$, there exists $i\in\{1,2,3\}$ with $[\alpha_k^i][\Sigma]\not=0$. After a hyperk\"ahler rotation, we can assume that $[\alpha_k^i][\Sigma]\not=0$ for any $k=0,1$, any $i=1,2,3$ and any $[\Sigma]\in H_2(M,\mathbb{Z})$ with $[\Sigma]^2=-2$.

  Now we can connect $[\alpha_0^i]$ with $[\alpha_1^i]$ by several pieces of segments. Along each segment, two of $[\alpha^i]$ are fixed while the remaining one is varying. We require that the actions of the two fixed $[\alpha^i]$ on any $[\Sigma]\in H_2(M,\mathbb{Z})$ with $[\Sigma]^2=-2$ are nonzero. Therefore along the path, the two conditions of Theorem \ref{Main-Theorem-3} are always satisfied.

  So we only need to consider each segment. Without loss of generality, we can assume that there is only one segment and $[\alpha^2]$, $[\alpha^3]$ are fixed along the segment. Actually, we can assume that $I, \omega^2$ and $\omega^3$ are invariant along the continuity path. Only $[\alpha^1]$, i.e. the $I$- K\"ahler class is varying. We denote the original $\omega_0^1\in[\alpha_0^1]$ by $\omega_0$. We will use it as the background metric.

  By Proposition 6.16 of \cite{Melrose}, for $\omega_0$, the second cohomology group $H^2(M,\mathbb{R})$ is naturally isomorphic to the space of bounded harmonic forms which are asymptotic to the linear combinations of $\mathrm{d}\theta^2\wedge\mathrm{d}\theta^3$, $\mathrm{d}\theta^3\wedge\mathrm{d}\theta^1$ and $\mathrm{d}\theta^1\wedge\mathrm{d}\theta^2$. We only care about the forms whose integrals on $F_{jk}$ are 0. Such kind of forms must decay exponentially. By the calculation in Theorem \ref{Ker-and-Coker-of-P} and the maximal principle, the self-dual part of any decaying harmonic form vanishes. It's well known that any anti-self-dual form must be (1,1).

  Thus, we can add linear combinations of those exponential decay anti-self-dual harmonic forms to change the K\"ahler class. However, the integrability condition $\int_M((\alpha^1)^2-(\omega^2)^2)=0$ may not be satisfied. Fortunately, there is an exponential decay exact form $\mathrm{d}((1-\chi(r-R-2))I^*\mathrm{d}r)$ on $M$. Moreover, it's (1,1) since in local coordinates $$\mathrm{d}((1-\chi(r-R-2))I^*\mathrm{d}r)=\mathrm{d}((1-\chi(r-R-2))(ir_j\mathrm{d}z^j-ir_{\bar j}\mathrm{d}\bar z^j))$$
  $$=-2i(1-\chi(r-R-2))r_{j\bar k}\mathrm{d}z^j\wedge\mathrm{d}\bar z^k+2i\chi'(r-R-2)r_j r_{\bar k}\mathrm{d}z^j\wedge\mathrm{d}\bar z^k.$$
  If we add this term with $\alpha^1$, then
  $$\int_M(\alpha^1+a\mathrm{d}((1-\chi(r-R-2))I^*\mathrm{d}r))^2-(\alpha^1)^2
  =2a\lim_{R\rightarrow\infty}\int_{r=R}I^*\mathrm{d}r\wedge\alpha^1.$$
  The integral $\int_{r=R}I^*\mathrm{d}r\wedge\alpha^1$ on $M$ converges to the term $\int_{\mathbb{T}^3} -\mathrm{d}\theta^1\wedge\mathrm{d}\theta^2\wedge\mathrm{d}\theta^3$ on the flat model, which is non-zero. So we can choose a suitable $a$ to achieve the integrability condition. We call the resulting (1,1) form $\alpha_t$. It satisfies the following conditions:

  (1) For any $m\ge0$, $||e^{\lambda_1r}\nabla^m_{\omega_0}(\alpha_t-\alpha_T)||_{C^0}$ converges to 0 when $t$ goes to $T$. In particular, $||e^{\lambda_1r}\nabla^m_{\omega_0}(\alpha_t-\omega_0)||_{C^0}$ is uniformly bounded.

  (2) $\int_M(\alpha_t^2-\omega_0^2)=0$.

  \begin{remark}
  $\alpha_t$ is positive in far enough region. However, it may not be positive in the compact part. That's the reason why the geometric existence part of \cite{HaskinsHeinNordstrom} fails.
  \end{remark}

  Now define $I$ as the set $$\{t\in[0,1]|\exists\phi_t s.t. \forall m\ge 0,|\nabla_{\omega_0}^m\phi_t|=O(e^{-\lambda_1r}),\omega_t=\alpha_t+i\partial\bar\partial\phi_t>0, \omega_t^2=\omega_0^2\}.$$
  It's trivial that $0\in I$.

  \begin{theorem}
  $I$ is open
  \end{theorem}
  \begin{proof}
  Suppose $T\in I$, then as long as $t$ is close enough to $T$,  $\alpha_t+i\partial\bar\partial\phi_T$ is positive. It satisfies the integrability condition $$\int_M((\alpha_t+i\partial\bar\partial\phi_T)^2-\omega_0^2)=
  \int_M((\alpha_t+i\partial\bar\partial\phi_T)^2-\alpha_t^2)
  +\int_M(\alpha_t^2-\omega_0^2)=0.$$
  By Theorem 4.1 of \cite{HaskinsHeinNordstrom}, $(\alpha_t+i\partial\bar\partial\phi_T+i\partial\bar\partial\phi)^2=\omega_0^2$ has a solution $\phi$. So $t\in I$ with $\phi_t=\phi_T+\phi$.
  \end{proof}

  Now we are going to show that $I$ is closed. Assume that $\{t_i\}\in I$ converge to $T$. To make the notation simpler, we will use $\alpha_i$, $\omega_i$ and $\phi_i$ to denote $\alpha_{t_i}$, $\omega_{t_i}$ and $\phi_{t_i}$.

  We start from an estimate:
  \begin{theorem}
  $$\int_M(\mathrm{tr}_{\omega_0}\omega_i-2)\frac{\omega_0^2}{2}\le C.$$
  Moreover $$\int_M(\mathrm{tr}_{\omega_j}\omega_i-2)\frac{\omega_j^2}{2}\rightarrow 0$$ as $i,j\rightarrow\infty$.
  \label{Energy-bound}
  \end{theorem}
  \begin{proof}
  $$\int_M(\mathrm{tr}_{\omega_0}\omega_i-2)\frac{\omega_0^2}{2}
  =\int_M\omega_0\wedge\omega_i-\omega_0^2
  =\int_M\omega_0\wedge(\alpha_i-\omega_0)\le C.$$
  Moreover $$\int_M(\mathrm{tr}_{\omega_j}\omega_i-2)\frac{\omega_j^2}{2}=\int_M\omega_j\wedge\omega_i-\omega_j^2
  =\int_M\alpha_j\wedge(\alpha_i-\alpha_j)\rightarrow 0$$ as $i,j\rightarrow\infty$.
  \end{proof}
  \begin{remark}
  By mean inequality, both $\mathrm{tr}_{\omega_0}\omega_i-2$ and $\mathrm{tr}_{\omega_j}\omega_i-2$ are non-negative since $\omega_0^2=\omega_i^2=\omega_j^2$.
  \end{remark}

  \begin{theorem}
  Let $U_N$ be the sets $\{N\le r\le N+1\}$ in the sense of $\omega_0$. Then for all large enough $N$, there exist subsets $V_{Ni}\subset U_N$ such that the volume $\mathrm{Vol}(V_{Ni})\ge \mathrm{Vol}(U_N)/2\ge C$ and for any $y_1,y_2\in V_{Ni}$, $d_{\omega_i}(y_1,y_2)\le C_1$.
  \label{Existence-of-V(N)}
  \end{theorem}
  \begin{proof}
  It was proved by Demailly, Peternell and Schneider as Lemma 1.3 of \cite{DemaillyPeternellSchneider} from the bound in Theorem \ref{Energy-bound}.
  \end{proof}

  By the volume comparison theorem on Ricci flat manifolds, if we pick any point $p_{Ni}\in V_{Ni}$, then the volume of radius $R$ ball centered at $p_{Ni}$ in the sense of $\omega_i$ has a uniform lower bound depending on $R$.

  \begin{theorem}
  For any fixed number $R$, the $\omega_i$-curvature in $B_{\omega_i}(p_{Ni},R)$ is uniformly bounded. Moreover, the $\omega_i$-holomorphic radius in $B_{\omega_i}(p_{Ni},R)$ is uniformly bounded below.
  \label{Curvature-bound}
  \end{theorem}
  \begin{proof}
  Suppose on the contrary, the $\omega_i$-curvature goes to infinity. Then we can rescale the metric so that the largest curvature equals to 1. By Theorem 4.7 of \cite{CheegerGromovTaylor}, the volume lower bound and the curvature bound imply the lower bound on the injectivity radius. Then, by Lemma 4.3 of \cite{Ruan}, the holomorphic radius has a lower bound. By Page 483 of \cite{Anderson}, the bound on the $L^2$-norm of curvature, the lower bound on the volume and the harmonic radius imply that the rescaled metric converges to an Einstein ALE space $M_\infty$. Replacing the harmonic radius by the holomorphic radius, we can show that $M_\infty$ is actually K\"ahler. Moreover, before taking limit, the manifold has a parallel holomorphic symplectic form $\omega^2+i\omega^3$. Thus, on $M_\infty$, there exists a parallel holomorphic symplectic form, too. In other words, $M_\infty$ is actually hyperk\"ahler.

  By Bando-Kasue-Nakajima \cite{BandoKasueNakajima} and Kronheimer \cite{Kronheimer1} \cite{Kronheimer2}, the (non-flat) ALE-gravitational instanton $M_\infty$ contains a curve $\Sigma_\infty$ with self intersection number -2. Before rescaling, the integrals of $\omega_i$, $\omega^2$ and $\omega^3$ on $\Sigma_i$ converge to 0.

  Recall that $H_2(M,\mathbb{R})$ is generated by 8 curves $\Sigma_\alpha$ and 3 faces $F_{23},F_{31},F_{12}$. In fact, similar to Section B of \cite{SchulzTammaro}, any element in $H_2(M,\mathbb{Z})$ can be represented by half integral linear combinations of $\Sigma_\alpha$ and $F_{\alpha\beta}$. Let
  $$[\Sigma_i]=\frac{1}{2}\sum(m_{i\alpha}[\Sigma_\alpha]+m_{i\alpha\beta}[F_{\alpha\beta}]).$$ Then $[\Sigma_i]^2=-2=\frac{-2}{4}\sum m_{i\alpha}^2$. So there are only finitely many possibilities of $m_{i\alpha}$. By condition (1) of Theorem \ref{Main-Theorem-3}, the actions of $\lim_{i\rightarrow\infty}[\omega_i]$, $[\omega^2]$ and $[\omega^3]$ on $F_{\alpha\beta}$ are linearly independent. Since the integrals of $\omega_i$, $\omega^2$ and $\omega^3$ on $\Sigma_i$ converge to 0, we know that $m_{i\alpha\beta}$ also has a uniform bound. In other words, the homology class $[\Sigma_i]$ only has finitely many possibilities. Taking a subsequence where the homology class of $\Sigma_i$ are same, we obtain a contradiction to the condition (2) of Theorem \ref{Main-Theorem-3}.

  We've obtained a bound on the curvature. Theorem 4.7 of \cite{CheegerGromovTaylor} and Lemma 4.3 of \cite{Ruan} now provide a lower bound on the holomorphic radius.
  \end{proof}

  Let $D_0$ be the upper bound on the diameter of $U_N$ with respect to $\omega_0$. We are interested in the function $e(t_i)=\mathrm{tr}_{\omega_i}\omega_0=\mathrm{tr}_{\omega_0}\omega_i$ on $B_{\omega_i}(p_{Ni},10D_0)$. We start from a theorem:

  \begin{theorem}
  There exists a constant $C_2$ such that if
  $$\gamma^{-2}\int_{B_{\omega_i}(p,\gamma)}e(t_i)\le C_2$$
  for some ball $B_{\omega_i}(p,\gamma)\subset B_{\omega_i}(p_{Ni},10D_0)$, then
  $$\sup_{\sigma\in[0,\frac{2}{3}\gamma]}\sigma^2\sup_{B_{\omega_i}(p,\frac{2}{3}\gamma-\sigma)}e(t_i)
  \le\frac{1}{C_2}\gamma^{-2}\int_{B_{\omega_i}(p,\gamma)}e(t_i)\le1.$$
  \label{Ruan}
  \end{theorem}
  \begin{proof}
  It's well known that in the $C^{1,\alpha}$-holomorphic radius, there are higher derivative bounds automatically. Thus, the constant in Proposition 2.1 of \cite{Ruan} is uniform. (There are several errors in \cite{Ruan}. After correcting them, we can only get ``$\frac{2}{3}\gamma$" instead of ``$\gamma$" in the statement)
  \end{proof}

  \begin{theorem}
  For each $i$, there exists a set $A_i\subset\{0,1,2...\}$ such that for all $N\not\in A_i$, $\sup_{B_{\omega_i}(p_{Ni},4D_0)}e(t_i)\le2.5$ and the number of elements in $A_i$ is bounded.
  \label{Pointwise-energy-bound}
  \end{theorem}
  \begin{proof}
  For any point $p\in M$, the number of $N$ such that $p\in B_{\omega_i}(p_{Ni},10D_0)$ is bounded. Actually, by Theorem \ref{Existence-of-V(N)}, for all such $N$, the set $V_{Ni}$ is contained in $B_{\omega_i}(p,10D_0+C_1)$. The volume of $B_{\omega_i}(p,10D_0+C_1)$ has an upper bound by volume comparison, while the volumes of the disjoint sets $V_{Ni}$ have a lower bound.

  Therefore, by Theorem \ref{Energy-bound} $$\sum_{N}\int_{B_{\omega_i}(p_{Ni},10D_0)}(e(t_i)-2)\le C.$$

  Let $\gamma$ be a constant smaller than $D_0$ such that $\gamma^2\le\frac{C_2}{2\pi^2}$. Then by the volume comparison theorem, the volume of $B(R)$ on a Ricci flat 4-manifold is bounded by $\frac{\pi^2}{2}R^4$. So
  $$\int_{B_{\omega_i}(p,\gamma)}2=2\mathrm{Vol}(B_{\omega_i}(p,\gamma))\le\frac{C_2}{2}\gamma^2.$$
  Therefore, as long as $$\int_{B_{\omega_i}(p_{Ni},10D_0)}(e(t_i)-2)\le\frac{C_2}{2}\gamma^2,$$ we can get $$\int_{B_{\omega_i}(p,\gamma)}e(t_i)\le C_2\gamma^2$$ for all $p\in B_{\omega_i}(p_{Ni},8D_0)$.
  So $e(t_i)=\mathrm{tr}_{\omega_i}\omega_0\le\frac{9}{4}\gamma^{-2}$ in $B_{\omega_i}(p_{Ni},8D_0)$ by Theorem \ref{Ruan}. In particular, $B_{\omega_i}(p_{Ni},8D_0)\subset B_{\omega_0}(p_{Ni},18\gamma^{-2}D_0)$.

  For any $\epsilon\le\frac{C_2}{2}\gamma^2$, let $A_{i,\epsilon}$ denote the set of $N$ such that the integral $\int_{B_{\omega_i}(p_{Ni},10D_0)}(e(t_i)-2)>\epsilon$ or $\sup_{B_{\omega_0}(p_{Ni},18\gamma^{-2}D_0)}|Rm(\omega_0)|>\epsilon$. Then the number of elements in $A_{i,\epsilon}$ is bounded. For all $N\not\in A_{i,\epsilon}$, it's well known \cite{Ruan} that $$-\Delta_{\omega_i}e(t_i)\le C |Rm(\omega_0)|e(t_i)^2\le C|Rm(\omega_0)|$$ in $B_{\omega_i}(p_{Ni},8D_0)$.  So both $\sup_{B_{\omega_i}(p_{Ni},8D_0)}-\Delta_{\omega_i}e(t_i)$ and $\int_{B_{\omega_i}(p_{Ni},8D_0)}(e(t_i)-2)$ are bounded by $C\epsilon$. By Theorem 9.20 of \cite{GilbargTrudinger}, $\sup_{B_{\omega_i}(p_{Ni},4D_0)}(e(t_i)-2)\le C\epsilon$ for all $N\not\in A_{i,\epsilon}$. After a suitable choice of $\epsilon$, we can make it smaller than $1/2$.
  \end{proof}

  \begin{lemma}
  There exists a constant $R$ such that $M\subset\cup_{N\not\in A_i}\overline{B_{\omega_i}(p_{Ni},R)}$.
  \label{Diameter-bound}
  \end{lemma}

  \begin{proof}
  For all $N\not\in A_i$, $\frac{1}{2}\omega_i\le\omega_0\le2\omega_i$ in $B_{\omega_i}(p_{Ni},4D_0)$ by Theorem \ref{Pointwise-energy-bound}. In particular, $$U_{N}\subset\overline{B_{\omega_0}(p_{Ni},D_0)}\subset B_{\omega_i}(p_{Ni},4D_0).$$

  Let $U_{A_i}=\cup_{N\in A_i}U_N$. Suppose $R=\sup_{q\in U_{A_i}}\inf_{N\not\in A_i}d_{\omega_i}(p_{Ni},q)$
  is achieved by $q_i$ and $N_i$. Then we will use the argument similar to Theorem 3.1 of \cite{RuanZhang} and Theorem 3.1 of \cite{Tosatti} proved by applying Theorem I.4.1 of \cite{SchoenYau}. Actually, if $R>10D_0$, by Theorem \ref{Pointwise-energy-bound}, it's easy to see that $B_{\omega_i}(q_i,R-10D_0)\subset U_{A_i}$ and $U_{N_i}\subset B_{\omega_i}(q_i,R+10D_0)\setminus B_{\omega_i}(q_i,R-10D_0)$. Since the volume of $U_{A_i}$ is bounded from above and the volume of $U_{N_i}$ is bounded from below, it's easy to get a bound on $R$ by the Bishop-Gromov volume comparison theorem.
  \end{proof}

  \begin{theorem}
  $e(t_i)=\mathrm{tr}_{\omega_0}\omega_i=\mathrm{tr}_{\omega_i}\omega_0$ is uniformly bounded on $M$.
  \end{theorem}
  \begin{proof}
  By Theorem \ref{Curvature-bound} and Lemma \ref{Diameter-bound}, the $\omega_i$-holomorphic radius is bounded from below. So the constant in Theorem \ref{Ruan} is uniform if we replace $\omega_0$ by $\omega_j$ in the statement of Theorem \ref{Ruan}. By Theorem \ref{Energy-bound}, $\int_M(\mathrm{tr}_{\omega_j}\omega_i-2)\frac{\omega_j^2}{2}\rightarrow 0$ as $i,j\rightarrow\infty$. So for large enough $i$ and $j$, $\mathrm{tr}_{\omega_j}\omega_i$ is uniformly bounded on $M$. Fix $j$ and let $i$ go to infinity. Since $C_j^{-1}\omega_0\le\omega_j\le C_j\omega_0$, the bound on $\mathrm{tr}_{\omega_j}\omega_i$ automatically implies a bound on $\mathrm{tr}_{\omega_0}\omega_i$.
  \end{proof}

  Now we are ready to use the arguments in \cite{HaskinsHeinNordstrom} to prove Theorem \ref{Main-Theorem-3}. Let $N$ be a large constant such that when $r\ge N$, $\frac{1}{2}\omega_0\le\alpha_i\le2\omega_0$. We start from a theorem which can be easily deduced from Proposition 4.21 of \cite{HaskinsHeinNordstrom}:

  \begin{theorem}
  Let $w=\frac{e^{-\delta r}}{\int_Me^{-\delta r}}$ be a weight function. Define the weighted norm $||u||_{L^p_w}$ by
  $||u||_{L^p_w}=(\int|u|^p w)^{1/p}$, then for all $u\in C^{\infty}_0$,
  $$||u||_{L^4_w(\{r\ge N\})}\le C||\nabla u||_{L^2(\{r\ge N\})}+C||u||_{L^4(\{N\le r\le N+1\})}.$$
  It's easy to see that for all $1\le p\le q\le\infty$, $||u||_{L^p_w}\le ||u||_{L^q_w}$ by H\"older's inequality.
  \label{Weighted-Sobolev-embedding}
  \end{theorem}

  \begin{theorem}
  $I$ is closed.
  \end{theorem}
  \begin{proof}
  Let $\phi_{ai}=\frac{\int_{M}\phi_{i}e^{-2\delta r}}{\int_{M}e^{-2\delta r}}$ be the weighted average of $\phi_i$. By the standard Lockhart-McOwen theory \cite{LockhartMcOwen}, since constant is the only harmonic function less than $e^{\delta r}$, we can obtain a bound on $||e^{-\delta r}(\phi_i-\phi_{ai})||_{W^{2,2}}$ from the $L^2$ bound of $e^{-\delta r}\Delta_{\omega_0}\phi_i=e^{-\delta r}(\mathrm{tr}_{\omega_0}\omega_i-\mathrm{tr}_{\omega_0}\alpha_i)$.

  Let $u_i=\phi_i-\phi_{ai}$. We already obtain a bound on $||u_i||_{W^{2,2}(\{r\le N+4\})}$ and $||\Delta_{\omega_0}u_i||_{L^\infty(M)}$. So $||u_i||_{W^{2,p}(\{r\le N+3\})}$ is bounded for any $p\in(1,\infty)$ by Theorem 9.11 of \cite{GilbargTrudinger}.

  The $C^{2,\alpha}$-estimate for real Monge-Amp\`ere equation was done by Evans-Krylov-Trudinger. See Section 17.4 of \cite{GilbargTrudinger} for details. Now we are in the complex case. However, the arguments in Section 17.4 of \cite{GilbargTrudinger} still work. An alternative way to achieve the bound on $[\partial\bar\partial u_i]_{C^{\alpha}(\{r\le N+2\})}$ for all $0<\alpha<1$ was done by Theorem 1.5 of \cite{ChenWang} using the rescaling argument. Now it's standard to get a $C^{\infty}$ bound of $u_i$ on $\{r\le N+1\}$ through Schauder estimates.

  As in \cite{HaskinsHeinNordstrom},
  $$\begin{array}{lcl}& &\int_{r\ge N}|\nabla|u_i|^{\frac{p}{2}}|^2\alpha_i^2\\
  &\le&\frac{p^2}{p-1} [\int_{r\ge N}u_i|u_i|^{p-2}(\omega_i^2-\alpha_i^2)
  -\frac{1}{2}\int_{r=N}u_i|u_i|^{p-2}\mathrm{d}^{c}u_i\wedge(\omega_i+\alpha_i)].
  \end{array}$$
  Therefore, for $p\ge 2$,
  $$\int_{r\ge N}|\nabla|u_i|^{\frac{p}{2}}|^2\le Cp(\int_{r\ge N}|u_i|^{p-1}w+C_3^{p-1})\le Cp(C_3\int_{r\ge N}|u_i|^{p-1}w+C_3^p),$$
  where $C_3$ is a bound on $\sup_{\{r\le N+1\}}|u_i|$. By Young's inequality,
  $$\int_{r\ge N}|\nabla|u_i|^{\frac{p}{2}}|^2\le Cp^2(||u_i||_{L^{p-1}_w(\{r\ge N\})}^p+C_3^p)\le Cp^2(||u_i||_{L^p_w(\{r\ge N\})}^p+C_3^p).$$
  Apply Theorem \ref{Weighted-Sobolev-embedding} to $|u_i|^{p/2}$. Then
  $$||u_i||_{L^{2p}_w(\{r\ge N\})}^{2p}\le C_4p^4(||u_i||_{L^p_w(\{r\ge N\})}^{2p}+C_3^{2p}).$$
  We already know that $||u_i||_{L^{2}_w(\{r\ge N\})}\le C_5$. That's our starting point. We are going to obtain a bound on
  $||u_i||_{L^{\infty}(\{r\ge N\})}=\lim_{j\rightarrow\infty}||u_i||_{L^{2^j}_w(\{r\ge N\})}$ by Moser iteration.

  (1) If $||u_i||_{L^{2^j}_w(\{r\ge N\})}\le C_3$ for all $j\ge 1$, then $||u_i||_{L^{\infty}(\{r\ge N\})}\le C_3$.

  (2) If $||u_i||_{L^{2^j}_w(\{r\ge N\})}\le C_3$ for all $1\le j\le k$ but $||u_i||_{L^{2^j}_w(\{r\ge N\})}>C_3$ for all $j>k$, then
  $$||u_i||_{L^{\infty}(\{r\ge N\})}\le (2C_4)^{\sum_{j=k}^{\infty}2^{-j-1}}2^{\sum_{j=k}^{\infty}2^{-j+1}j}C_3\le C$$

  (3) If $||u_i||_{L^{2^j}_w(\{r\ge N\})}>C_3$ for all $j\ge 1$, then $$||u_i||_{L^{\infty}(\{r\ge N\})}\le (2C_4)^{\sum_{j=1}^{\infty}2^{-j-1}}2^{\sum_{j=1}^{\infty}2^{-j+1}j}C_5\le C$$

  The $L^\infty$ bound on $u_i=\phi_i-\phi_{ai}$ implies a bound on $\phi_{ai}$ since $\phi_i$ decay exponentially. Therefore, we actually have a $L^\infty$ bound on $\phi_i$. Then we can obtain a global $C^{\infty}$ bound as before. Finally, we can go through the Step 3 and Step 4 in \cite{HaskinsHeinNordstrom} to get the $C^{\infty}$ bound on $e^{\lambda_1r}\phi_i$. We are done by taking the limit of some subsequence of $\{\phi_i\}$.
  \end{proof}

\end{document}